\documentclass[11pt,reqno]{amsart}

\usepackage[T1]{fontenc}

\usepackage{amsmath,xypic}								
\usepackage{amssymb}
\usepackage{amsthm}
\usepackage{amscd}
\usepackage{amsfonts}
\usepackage{mathabx}
\usepackage{stmaryrd}
\usepackage[all]{xy}

\usepackage{caption}
\usepackage{subcaption}
\usepackage{euler}

\usepackage{extarrows}

\usepackage[colorlinks, linktocpage, citecolor = purple, linkcolor = blue]{hyperref}
\usepackage{color}

\usepackage{tikz}									
\usetikzlibrary{matrix}
\usetikzlibrary{patterns}
\usetikzlibrary{positioning}
\usetikzlibrary{decorations.pathmorphing}
\usetikzlibrary{cd}

\usepackage{ytableau}
\usepackage{fullpage}
\usepackage[shortlabels]{enumitem}

\newtheorem{maintheorem}{Theorem}	
\newtheorem{maincorollary}[maintheorem]{Corollary}								
\newtheorem{theorem}{Theorem}[section]
\newtheorem{lemma}[theorem]{Lemma}
\newtheorem{proposition}[theorem]{Proposition}
\newtheorem{corollary}[theorem]{Corollary} 
\newtheorem{conjecture}[theorem]{Conjecture} 

\theoremstyle{definition}
								
\newtheorem{definition}[theorem]{Definition}
\newtheorem{example}[theorem]{Example}

\newtheorem{remark}[theorem]{Remark}
\newtheorem*{remark*}{Remark}

\numberwithin{theorem}{subsection}

\newcommand{\R}{\mathbb{R}}

\newcommand{\Z}{\mathbb{Z}}
\newcommand{\ZZ}{\mathbb{Z}}

\newcommand{\PP}{\mathbb{P}}
\newcommand{\NN}{\mathbb{N}}

\newcommand{\RR}{\mathbb{R}}

\newcommand{\G}{\mathbb{G}}
\newcommand{\GG}{\mathbb{G}}

\newcommand{\calA}{\mathcal{A}}
\newcommand{\calB}{\mathcal{B}}

\newcommand{\calO}{\mathcal{O}}

\newcommand{\calR}{\mathcal{R}}

\DeclareMathOperator{\Pic}{Pic}
\DeclareMathOperator{\Spec}{Spec}
\DeclareMathOperator{\Hom}{Hom}

\DeclareMathOperator{\trop}{trop}
\DeclareMathOperator{\val}{val}

\DeclareMathOperator{\Trop}{Trop}

\DeclareMathOperator{\pr}{pr}
\DeclareMathOperator{\id}{id}

\DeclareMathOperator{\Div}{Div}

\DeclareMathOperator{\Nm}{Nm}
\DeclareMathOperator{\PDiv}{PDiv}
\DeclareMathOperator{\Alb}{Alb}
\DeclareMathOperator{\coker}{coker}

\let\Rat\relax
\DeclareMathOperator{\Rat}{Rat}
\let\div\relax
\DeclareMathOperator{\div}{div}
\let\ord\relax
\DeclareMathOperator{\ord}{ord}
\let\Jac\relax
\DeclareMathOperator{\Jac}{Jac}
\let\im\relax
\DeclareMathOperator{\im}{im}

\title[Skeletons of Prym varieties and tropical Brill--Noether theory]{Skeletons of Prym varieties and Brill--Noether theory}

\author{Yoav Len}
\address{Mathematical Institute, University of St Andrews, St Andrews KY16 9SS, UK}
\email{\href{mailto:yoav.len@st-andrews.ac.uk}{yoav.len@st-andrews.ac.uk}}

\author{Martin Ulirsch}
\address{Institut f\"ur Mathematik, Goethe-Universit\"at Frankfurt, 60325 Frankfurt am Main, Germany}
\email{\href{mailto:ulirsch@math.uni-frankfurt.de}{ulirsch@math.uni-frankfurt.de}}


\begin{document}

\begin{abstract} 
We show that the non-Archimedean skeleton of the Prym variety associated to an unramified double cover of an  algebraic curve is naturally isomorphic (as a principally polarized tropical abelian variety) to the tropical Prym variety of the associated tropical double cover. This confirms a conjecture by Jensen and the first author. We prove a new upper bound on the dimension of the Prym--Brill--Noether locus for a generic unramified double cover in a dense open subset in the moduli space of unramified double covers of curves with fixed even gonality on the base. Our methods also give a new proof of the classical Prym--Brill--Noether Theorem for generic unramified double covers that is originally due to Welters and Bertram.
\end{abstract}

\maketitle

\setcounter{tocdepth}{1}
\tableofcontents


\section*{Introduction}

Prym varieties are a class of abelian varieties that are associated to covers of Riemann surfaces. They form a bridge between the geometry of curves and the geometry of abelian varieties, and provide a rare class of abelian varieties that may be exhibited explicitly. 

Let $X$ be smooth projective curve and let $\pi\colon \widetilde{X}\rightarrow X$ be an unramified double cover. The map $\pi$ induces natural \emph{norm homomorphism} 
\begin{equation*}
\Nm_\pi\colon \Pic(\widetilde{X})\longrightarrow\Pic(X)
\end{equation*}
 given by pushing forward divisors, i.e. by $\calO_{\widetilde{X}}(\widetilde{D})\mapsto\calO_X(\pi_\ast \widetilde{D})$ for all divisors $D$ on $X$. The kernel of $\Nm_\pi$ is a subgroup of $\Pic_0(X)$ consisting of two components. The component containing the identity is known as the \emph{Prym variety} $\Pr(X,\pi)$ associated with the unramified double cover. As explained in \cite{Mumford_Prym}, it carries a natural principal polarization, whose theta divisor $\Xi$ fulfills 
 \begin{equation*}
     i^\ast \widetilde{\Theta} =2\cdot \Xi,
 \end{equation*}
 where $i^\ast \widetilde{\Theta}$ denotes the pullback of the theta divisor on $\Jac(\widetilde{X})=\Pic_0(\widetilde{X})$.
 
 Fixing a point $q\in \widetilde{X}(K)$, there is a Prym theoretic analogue of the Abel--Jacobi map, known as the \emph{Abel--Prym} map $\alpha_{X,\pi}\colon \widetilde{X}\rightarrow \Pr(X,\pi)$. Explicitly, it is given by
\begin{equation*}
p\longmapsto i' \big(\calO_{\widetilde{X}}(p-q)\big) \,
\end{equation*}
where $i'$ denotes the dual homomorphism to the inclusion $i\colon\Pr(X,\pi)\hookrightarrow \Jac(\widetilde{X})$.

\subsection*{Tropical Prym varieties} In \cite[Section 6]{JensenLen_thetachars}, Jensen and the first author gave a tropical analogue of this construction: Let $\Gamma$ be a tropical curve and $\pi\colon\widetilde{\Gamma}\rightarrow\Gamma$ an unramified double cover. Again, this induces a natural \emph{tropical norm homomorphism}
\begin{equation*}
\Nm_\pi\colon\Pic(\widetilde{\Gamma})\longrightarrow \Pic(\Gamma)
\end{equation*}
given by pushing forward divisor classes, i.e. by $[\widetilde{D}]\mapsto \big[\pi_\ast \widetilde{D}\big]$. By Theorem \ref{thm_Prymvariety} below, the kernel of $\Nm_\pi$ has either one or two components and the component containing the identity carries a natural principal polarization; we say that it is the  \emph{tropical Prym variety $\Pr(\Gamma,\pi)$} associated to the unramified double cover $\pi\colon\widetilde{\Gamma}\rightarrow\Gamma$. Moreover, given a fixed point $q\in\widetilde{\Gamma}$, we may define a \emph{tropical Abel-Prym map} $\alpha_{\Gamma,\pi}\colon\widetilde{\Gamma}\rightarrow \Pr(\Gamma,\pi)$ by 
\begin{equation*}
p\longmapsto i' \big([p-q]\big)\ , 
\end{equation*}
where, again, $i'$ denotes the dual homomorphism to the inclusion $i\colon\Pr(\Gamma, \pi)\hookrightarrow\Jac(\widetilde{\Gamma})$.

 \subsection*{Skeletons of Prym varieties} Our first result is that the Prym construction behaves well with respect to tropicalization. Suppose that both $X$ and $\pi$ are defined over a non-Archimedean field $K$, i.e. a field that is complete with respect to a non-Archimedean absolute value. Let $\Gamma_X$ be the dual tropical curve of $X$. Then $\Gamma_X$ is naturally identified with the non-Archimedean skeleton of $X^{an}$ and there is a natural strong deformation retraction $\rho_X\colon X^{an}\rightarrow \Gamma_X$. 
 
 On the other hand, given an abelian variety  $A$ with split semistable reduction over $K$, by \cite{Berkovich_book}, there is a natural strong deformation retraction $\rho_A\colon A^{an}\rightarrow \Sigma(A)$ from $A^{an}$ onto a closed subset $\Sigma(A)$ of $A^{an}$ that has the structure of a tropical abelian variety, the \emph{non-Archimedean skeleton} of $A^{an}$. A (principal) polarization on $A$ naturally induces a (principal) polarization on $\Sigma(A)$.

\begin{maintheorem}\label{thm_skeletonofPrym=Prymofskeleton}
There is a canonical isomorphism
\begin{equation*}
\mu_{X,\pi}\colon\Pr(\Gamma_X, \pi^{trop}) \xlongrightarrow{\simeq} \Sigma\big(\Pr(X,\pi)\big)
\end{equation*}
of principally polarized tropical abelian varieties that commutes with the Abel-Prym maps, i.e. for which the natural diagram
\begin{center}\begin{tikzcd}
X^{an}\arrow[rrr,"\rho_X"]\arrow[d,"\alpha_{X,\pi}^{an}"']&&& \Gamma_X\arrow[d,"\alpha_{\Gamma,\pi^{trop}}"]\\
\Pr(X,\pi)^{an}\arrow[rr, "\rho_{\Pr(X,\pi)}"] && \Sigma\big(\Pr(X,\pi)\big) & \Pr(\Gamma_X, \pi^{trop})\arrow[l,"\sim"',"\mu_{X,\pi}"]
\end{tikzcd}\end{center}
commutes. 
\end{maintheorem}

In \cite{BakerRabinoff_skelJac=Jacskel}, Baker and Rabinoff show that the non-Archimedean skeleton $\Sigma\big(\Jac(X)\big)$ of the Jacobian $\Jac(X)^{an}$ is naturally isomorphic (as a principally polarized tropical abelian variety) to the tropical Jacobian $\Jac(\Gamma_X)$ of the dual tropical curve of $X$. With Theorem \ref{thm_skeletonofPrym=Prymofskeleton} we expand on their result and thereby confirm \cite[Conjecture 6.3]{JensenLen_thetachars}. We emphasize that our definition of $\Pr(X,\pi)$ differs slightly from the one in \cite{JensenLen_thetachars},  as we work with the Jacobian of the underlying metric graph, rather than the augmented tropical Jacobian.

\subsection*{Prym--Brill--Noether theory} Let $X$ be a smooth projective curve of genus $g$ and $\pi\colon \widetilde{X}\rightarrow X$ an unramified double cover. Rather than working with the components of the kernel of the norm map $\Nm_\pi$, 
it is occasionally more convenient to consider the  preimage $\Nm_\pi^{-1}(\omega_X)$ of the canonical line bundle in $\Pic_{2g-2}(\widetilde{X})$. Since the components parameterizing line bundles of positive or negative parity are naturally a torsor over $\Pr(X,\pi)$, the above results  transfer to this setting. 

This point of view paves the way to studying Brill--Noether loci in Prym varieties. Fix $r\geq -1$. In \cite{Welters_GiesekerPetri}, Welters defines the \emph{Prym--Brill--Noether locus} $V^r(X,\pi)$ to be the closed subset 
\begin{equation*}
V^r(X,\pi)=\big\{L\in\Pic_{2g-2}(\widetilde{X})\big\vert \Nm_\pi(L)=\omega_X, r(L)\geq r \textrm{ and } r(L) \equiv r \pmod 2\big\}
\end{equation*}
in $\Pic_{2g-2}(\widetilde{X})$. Bertram's existence theorem for Prym special divisors \cite[Theorem 1.4]{Bertram_existenceforPrymspecialdivisors} (also see \cite[Theorem 9]{DeConciniPragacz}) shows that this locus is non-empty, as long as $g-1-\binom{r+1}{2}\geq 0$. An elementary estimate (see \cite[Proposition 1.4]{Welters_GiesekerPetri}) then shows that 
 \begin{equation}\label{eq_lowerbound}
 \dim V^r(X,\pi)\geq g-1-\binom{r+1}{2}
 \end{equation}
 for all curves $X$ of genus $g$ and all unramified double covers $\pi\colon \widetilde{X}\rightarrow X$. Using these two facts, Welters' Prym-Gieseker-Petri Theorem \cite[Theorem 1.11]{Welters_GiesekerPetri} implies that, for a general unramified double cover, inequality \eqref{eq_lowerbound} is an equality, namely a \emph{Brill--Noether theorem for double covers}.

\subsection*{Prym--Brill--Noether theory with gonality} In contrast, very little is known about special Prym curves. Using Theorem \ref{thm_skeletonofPrym=Prymofskeleton}, and expanding on the work of Pflueger \cite{Pflueger_kgonalcurves} (whose use of chains of loops builds of course on  \cite{CoolsDraismaPayneRobeva}), we find a previously unknown upper bound on the dimension of $V^r(X,\pi)$ for general unramified double covers of curves $X$ whose gonality is either even or sufficiently large. 

Denote by $\calR_g$  the moduli space of unramified double covers $\pi\colon \widetilde{X}\rightarrow X$ of a smooth projective curve of genus $g\geq 2$, as e.g. introduced in \cite{Beauville_Prym&Schottky}. We refer to the locus of unramified double covers $\pi\colon \widetilde{X}\rightarrow X$ for which $X$ has gonality $k\geq 2$ as the \emph{$k$-gonal locus} in $\calR_g$. 
For convenience, denote $\ell=\lceil\frac{k}{2}\rceil$. For $r\geq -1$, we write 
\begin{equation*}
  n=n(r,\ell) =
  \begin{cases}
                    \binom{\ell+1}{2}+\ell(r-\ell) &
                    \text{if $\ell\leq r-1$} \\
                    \binom{r+1}{2} & \text{if $\ell > r-1$.}
  \end{cases}
\end{equation*}

\begin{maintheorem}\label{thm_PrymBrillNoetherwithfixedgonality}
Suppose $k\geq 2$ is either even or greater than $2r-2$. There is a non-empty open subset in the $k$-gonal locus of $\calR_g$ such that for every unramified double cover  $\pi\colon\widetilde{X}\rightarrow X$ in this open subset we have:
\begin{equation} \label{eq_upperbound}
\dim V^r(X,\pi) \leq g-1 - n(r,\ell) \ .
\end{equation}
In particular, the Prym--Brill--Noether locus $V^r(X,\pi)$ is empty if $g-1 < n(r,\ell)$. 
\end{maintheorem}

\begin{remark*}
Following the completion of this manuscript, the first author together with Creech, Ritter, and Wu extended the theorem to unramified double covers of curves of any gonality  \cite[Corollary B]{CLRW_PBN}. 
\end{remark*}

As soon as $k\geq 3$, the moduli space of chains of covers $\widetilde{X}\rightarrow X\rightarrow \PP^1$, where the first arrow is an unramified double cover and the second arrow is a degree $k$ cover with simple ramification, is irreducible \cite[Theorem 2]{BiggersFried}; the $k$-gonal locus in $\calR_g$ is the closure of the image of this moduli space and is therefore irreducible when $k\geq 3$. So in this case the open subset in Theorem \ref{thm_PrymBrillNoetherwithfixedgonality} is dense and it makes sense to talk about a generic double cover $\pi\colon \widetilde{X}\rightarrow X$ in the $k$-gonal locus. 

For generic unramified double covers of curves $X$ of gonality $k\geq 2r-1$, the lower bound \eqref{eq_lowerbound} tells us that inequality \eqref{eq_upperbound} is, in fact, an equality.  Bertrams's  existence result \cite[Theorem 1.4]{Bertram_existenceforPrymspecialdivisors} implies that the emptiness criterion is necessary as well.

\begin{maincorollary}\label{cor_PrymBrillNoether}
Let $k\geq 2r-1$. There is a non-empty open subset in the $k$-gonal locus of $\calR_g$ such that for every unramified double cover in this open subset we have 
\begin{equation*}
    \dim V^r(X,\pi)=g-1 -\binom{r+1}{2} \ .
\end{equation*}
In particular, the Prym--Brill--Noether locus $V^r(X,\pi)$ is empty if and only if $g-1<\binom{r+1}{2}$. 
\end{maincorollary}

When $2k-2\geq g$, the general curve is $k$-gonal, so Corollary \ref{cor_PrymBrillNoether} coincides with the previously known Brill--Noether theorem for general double covers. 
However, the corollary extends the precise determination of the dimension to the range $4r-2\leq 2k < g+2$.

In \cite[Theorem 1.1 a)]{Hoering_PBN} H\"oring shows that for an arbitrary hyperelliptic base curve $X$ of genus $g\geq 6$ we have $\dim V^2(X,\pi)=g-3$ for all unramified double cover $\pi\colon \widetilde{X}\rightarrow X$. Theorem \ref{thm_PrymBrillNoetherwithfixedgonality} recovers this number as an upper bound for generic double covers $\pi\colon \widetilde{X}\rightarrow X$ by taking $r=2$ and $\ell=1$ in \eqref{eq_upperbound}. If $X$ is not hyperelliptic (and still of genus $g\geq 6)$, then it follows from \cite[Theorem 1.1 b)]{Hoering_PBN}  that $\dim V^2(X,\pi)=g-4$. We recover this equality for generic double covers by setting $r=2$ in Corollary \ref{cor_PrymBrillNoether}.

Our proof works in all characteristics prime to $2$ and $k$ and so, in particular, in characteristic zero. We expect inequality \eqref{eq_upperbound} to be an equality when $g\gg k$ 
and the emptiness condition to be necessary, even in the case of $X$ having even gonality $k< 2r$ (see \cite[Conjecture 3.9]{CLRW_PBN} and the surrounding discussion for more details). An adjusted version of the approach by Jensen and Ranganathan \cite{JensenRanganathan_BrillNoetherwithfixedgonality} using logarithmic stable maps to rational normal scrolls might provide  the desired lower bound. We hope to return to this part of the story in the near future. 

\subsection*{Further remarks and complements} 
 Let $\calA_g$ be the moduli space of principally polarized abelian varieties. There is a natural \emph{Prym-Torelli morphism} $\pr\colon\calR_g\rightarrow \calA_{g-1}$ that associates to an unramified double cover $\pi\colon \widetilde{X}\rightarrow X$ the associated Prym variety $\Pr(X,\pi)$. Theorem \ref{thm_skeletonofPrym=Prymofskeleton} says that the Prym-Torelli morphism  naturally commutes with tropicalization, i.e. that the diagram
\begin{center}\begin{tikzcd}
\calR_g^{an}\arrow[rr,"\trop_{\calR_g}"]\arrow[d,"\pr^{an}"']&&R_g^{trop}\arrow[d,"\pr^{trop}"]\\
\calA_{g-1}^{an}\arrow[rr,"\trop_{\calA_g}"] && A_{g-1}^{trop} 
\end{tikzcd}\end{center}
commutes. The reader may find a definition of the tropicalization map $\trop_{\calR_g}$ in \cite{CavalieriMarkwigRanganathan_admissiblecovers} and of $\trop_{\calA_g}$ in \cite{Viviani_tropicalTorelli}. We also refer to reader to \cite{CaporasoMeloPacini} for the closely related tropicalization map for the moduli spaces of curves with a theta characteristic.

There is a natural modular tropicalization map
\begin{equation*}
\trop_{X,\pi}\colon \Pr(X,\pi)^{an}\longrightarrow \Pr(\Gamma_X,\pi^{trop})
\end{equation*}
from the Berkovich space $\Pr(X,\pi)^{an}$ to $\Pr(\Gamma_X,\pi^{trop})$ that is induced by the pointwise tropicalization of $D$. The commutativity of the isomorphism in Theorem \ref{thm_skeletonofPrym=Prymofskeleton} with the Abel-Prym map implies that the diagram
\begin{equation*}
    \begin{tikzcd}
        \Pr(X,\pi) \arrow[rrd, bend left, "\trop_{X,\pi}"] \arrow[rd,"\rho_{\Pr(X,\pi)}"']&&\\
        & \Sigma\big(Pr(X,\pi)\big) & \Pr(\Gamma_X,\pi^{trop})\arrow[l,"\mu_{x,\pi}","\sim"']
    \end{tikzcd}
\end{equation*}
commutes. This allows us to apply both Baker's specialization inequality from \cite{Baker_specialization} and Gubler's Bieri-Groves Theorem \cite{Gubler_trop&nonArch} for abelian varieties to prove Theorem \ref{thm_PrymBrillNoetherwithfixedgonality}.


\subsection*{Acknowledgements} 
We thank Dmitry Zakharov for pointing out a gap in the proof of Lemma~\ref{lemma_spanningtrees}. 
We thank Matt Baker, Gavril Farkas, Martin M\"oller, Dave Jensen, Angela Ortega, Nathan Pflueger, and Dhruv Ranganathan  for insightful discussions. We also thank the referees for their helpful comments and remarks.
  M.U. acknowledges support from the LOEWE-Schwerpunkt ``Uniformi\-sierte Strukturen in Arithmetik und Geometrie''. This project  has  received  funding  from  the  European Union's Horizon 2020 research and innovation programme  under the Marie-Sk\l odowska-Curie Grant Agreement No. 793039. \includegraphics[height=1.7ex]{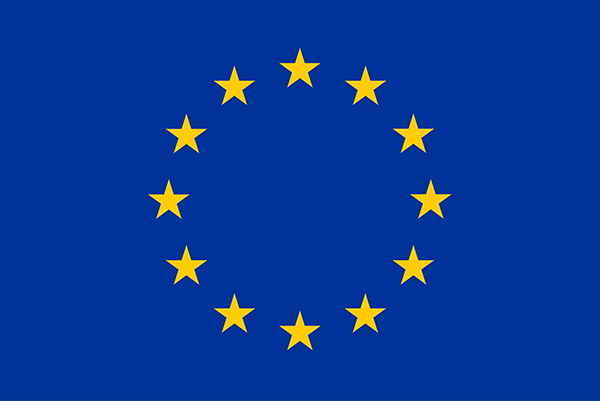}


\section{Tropical norm maps and the Prym variety} 

We begin by recalling the theory of tropical abelian varieties from \cite{FosterRabinoffShokriehSoto_tropicaltheta}, and  develop the tropical theory of Prym varieties expanding on \cite[Section 6]{JensenLen_thetachars}. See  \cite[Appendix B.1]{ACGHI} for the classical algebraic treatment. 


\subsection{Tropical abelian varieties}

Let $N$ be a free finitely generated abelian group. Set $N_\R=N\otimes_\ZZ \R$ and let $\Lambda\subseteq N_\R$ be a lattice of full rank. The quotient $\Sigma=N_\R/\Lambda$ is known as a \emph{real torus with integral structure}, where the term ``integral structure" refers to the choice of lattice $N\subseteq N_\R$, which  often differs from $\Lambda$ (see \cite[Section 2.1]{FosterRabinoffShokriehSoto_tropicaltheta}). Let $\Sigma_i=N^i_\R/\Lambda_i$ (for $i=1,2$) be two real tori with integral structure. There is a one-to-one correspondence between homomorphisms $f\colon \Sigma_1\rightarrow\Sigma_2$ of real tori and $\R$-linear homomorphisms $\widetilde{f}\colon N^1_\R\rightarrow N^2_\R$ such that $\widetilde{f}(\Lambda_1)\subseteq \Lambda_2$. We say that $f\colon\Sigma_1\rightarrow \Sigma_2$ is a \emph{homomorphism of real tori with integral structure} if $\widetilde{f}\colon N_\R^1\rightarrow N_\R^2$ is induced by a $\Z$-linear map $ N^1\rightarrow N^2$ (also denoted by $\widetilde{f}$). 

Let $M$ and $M'$ be two finitely generated free abelian groups of the same rank and let 
\begin{equation*}
\langle .,.\rangle \colon M'\times M\longrightarrow \R
\end{equation*}
be a non-degenerate pairing. We may think of $M'$ as a lattice in $N_\R$ via the embedding
$m'\mapsto \langle m',.\rangle\in\Hom(M,\R)=N_\R$ (and of $M$ as a lattice in $N_\R'$ via the embedding $m\mapsto\langle .,m\rangle\in\Hom(M',\R)=N_\R'$ respectively). We recall \cite[Definition 2.6]{FosterRabinoffShokriehSoto_tropicaltheta}:

\begin{definition}
The quotient $\Sigma=N_\R/M'$ is said to be a \emph{tropical abelian variety}, if there is a homomorphism $\lambda\colon M'\rightarrow M$ such that the bilinear form
\begin{equation*}\begin{split}
    \langle .,\lambda(.)\rangle \colon M'\times M'&\longrightarrow \R\\
    (m_1',m_2')&\longmapsto \langle m_1',\lambda(m_2')\rangle
\end{split}\end{equation*}
is symmetric and positive definite. 
\end{definition}

Write $\Sigma'=N_\R'/M$ and denote by $\widetilde{\phi}\colon N_\R\rightarrow N_\R'$ the map induced by $\lambda$. The map $\widetilde{\phi}$ takes $M'$ to $M$ and therefore induces a homomorphism $\phi\colon\Sigma\rightarrow \Sigma'$ of real tori with integral structure. We say that $\Sigma'$ is the \emph{dual tropical abelian variety} of $\Sigma$, and $\phi$ is called a \emph{polarization}. A polarization is said to be \emph{principal} if $\lambda$ is an isomorphism. 

A \emph{homomorphism} $f\colon \Sigma_1\rightarrow\Sigma_2$ of tropical abelian varieties is a homomorphism of real tori with integral structures. The \emph{dual homomorphism} is the unique homomorphism $f'\colon \Sigma_2'\rightarrow \Sigma_1'$ such that 
\begin{equation*}
    \big\langle\widetilde{f}(m_1'),m_2\big\rangle=\big\langle m_1',\widetilde{f}'(m_2)\big\rangle
\end{equation*}
for all $m_1'\in M'_1$ and $m_2\in M_2$. The association $f\mapsto f'$ defines a contravariant functor and we have a natural isomorphism $\Sigma\xrightarrow{\sim}\Sigma''$.

 Let $(\Sigma_i,\phi_i)$ be polarized tropical abelian varieties (for $i=1,2$) and $f\colon \Sigma_1\rightarrow \Sigma_2$ be a homomorphism. 
 Denote by $\ker(f)_0$ the connected component of the kernel of $f$ containing zero and by $\coker(f)$ the cokernel of $f$.
 
 \begin{proposition}\label{prop_tropker&coker}
 Both $\ker(f)_0$ and $\coker(f)$ are tropical abelian varieties. 
 \begin{enumerate}[(i)]
 \item The dual of $\ker(f)_0$ is $\coker(f')$ and  $\phi_1$ induces a polarization $i^\ast\phi_1\colon\ker(f)_0\rightarrow \coker(f')$. 
 \item The dual of $\coker(f)$ is $\ker(f')_0$ and  $\phi_2$ induces a polarization $q_\ast\phi_2\colon\coker(f)\rightarrow\ker(f')_0$.
 \end{enumerate}
 \end{proposition} 

So the category of polarized tropical abelian varieties an abelian category.
 Here $i$ refers to the inclusion $\ker(f)_0\hookrightarrow \Sigma_1$ and the letter $q$ to the quotient $\Sigma_2\rightarrow \Sigma_2/\im(f)$. We similarly write $i^\ast\lambda_1$ and $q_\ast\lambda_2$ the induced homomorphisms of integral structures.
 
 \begin{proof} The homomorphism $f\colon \Sigma_1\rightarrow\Sigma_2$ is induced by a linear homomorphism $\widetilde{f}\colon N^1_\R\rightarrow N_\R^2$ such that $\widetilde{f}(M_1')\subseteq M_2'$. Let $N_\R^K=\ker(\widetilde{f})$. Denote the restriction of $\widetilde{f}$ to a homomorphism $M_1'\rightarrow M_2'$ by $f_{M'}$ and let $M_K'=\ker(f_M)$. Then $\ker(f)_0$ is equal to the real torus with integral structure $N_\R^K/M_K'$. Similarly, we set $N_\R^C=\coker(\widetilde{f})$ and define $M_C'$ to be the quotient $M_2'/\im(f_{M'})^{sat}$ of $M_2'$ by the saturation of $\im(f_{M'})$ in $M_2'$. Then $\coker(f)$ is the torus $N_\R^C/M_C'$.
 Finally, the induced polarization $i^\ast\lambda_1$ is given by the composition
\begin{equation*}
    i^\ast\lambda_1\colon M'_K\xlongrightarrow{\widetilde{i}_{M'}} M_1'\xlongrightarrow{\lambda_1}M_1\xlongrightarrow{\widetilde{q}'} M_C \ .
\end{equation*}
This proves Part (i). Part (ii) follows from the dual argument and is left to the avid reader.
\end{proof}

We remark that, even when both $\phi_i$ (for $i=1,2$) are principal polarizations, the induced polarization $i^\ast\phi_1$ on $\ker(f)_0$ may not be a principal polarization (see Theorem \ref{thm_Prymvariety} below).


\subsection{Harmonic morphisms}\label{section_2:1} 

A \emph{metric graph} is a finite graph $G$ (possibly with loops and multiple edges) together with an edge length function $\ell \colon E(G)\rightarrow\R_{>0}$. We naturally associate to $(G,\ell)$ a metric space $\big\vert (G,\ell)\big\vert$,  by introducing an interval of length $\ell(e)$ for every edge $e$ in $G$, and gluing all of them according to the incidences of $G$. In this case, we say that the metric graph $(G,\ell)$ is a \emph{model} for the  metric space $\big\vert(G,\ell)\big\vert$. 
A \emph{tropical curve}  is a metric space $\Gamma$  together with a function $h\colon \Gamma\rightarrow\Z_{\geq 0}$ that is supported on the vertices of some model $(G,\ell)$ for $\Gamma$.

Let $\Gamma$ and $\widetilde{\Gamma}$ be tropical curves. A continuous map $\pi\colon \widetilde{\Gamma}\rightarrow\Gamma$ is called a \emph{morphism} if there exist models $(G,\ell)$ of $\Gamma$ and $(\widetilde{G},\widetilde{\ell})$ of $\widetilde{\Gamma}$ such that 
\begin{itemize}
    \item $\pi(V)\subseteq \widetilde{V}$,
    \item $\pi^{-1}(\widetilde{E})\subseteq E$, and
    \item the restriction to every edge $\widetilde{e}\in \widetilde{E}$ is a dilation by a factor $d_{\widetilde{e}}(\pi)\in\Z_{\geq 0}$. 
\end{itemize} 
We say that $\pi$ is \emph{finite} if $d_{\widetilde{e}}(\pi)>0$ for all edges $\widetilde{e}$ in $\widetilde{\Gamma}$.  

\begin{definition}\label{def_harmonicmorphism}
A finite morphism $f\colon \widetilde{\Gamma}\rightarrow\Gamma$ is said to be \emph{harmonic} at $\widetilde{p}\in\widetilde{\Gamma}$, if the sum
\begin{equation*}
    d_{\widetilde{p}}(\pi):=\sum_{\widetilde{v}\in T_{\widetilde{p}}\widetilde{\Gamma}, f(\widetilde{v})=v} d_{\widetilde{v}}(\pi)
\end{equation*}
does not depend on the choice of $v\in T_p\Gamma$. Here we write $T_p\Gamma$ (or $T_{\widetilde{p}}\widetilde{\Gamma}$) for the set of tangent directions emanating from $p$ (or $\widetilde{p}$ respectively). A finite morphism is called \emph{harmonic} if is surjective and harmonic at every point of $\widetilde{\Gamma}$. 
\end{definition}
For a harmonic morphism $\pi\colon\widetilde{\Gamma}\rightarrow\Gamma$, the number 
\begin{equation*}
    d_p(\pi)=\sum_{\widetilde{p}\in\widetilde{\Gamma}, \pi(\widetilde{p})=p} d_{\widetilde{p}}(\pi)
\end{equation*}
does not depend on the point $p\in\Gamma$ and is called the \emph{degree} of $\pi$. A harmonic morphism is said to be \emph{unramified} if the \emph{ramification divisor} $R_\pi=K_{\widetilde{\Gamma}}-f^\ast K_\Gamma$ is zero. Recall hereby that the \emph{canonical divisor} on a tropical curve $
\Gamma$ is given by
\begin{equation*}
   K_\Gamma=\sum_{p\in \Gamma}(2h(p)-2+\val(p))\cdot p \ .
\end{equation*}
We have $R_\pi=\sum_{\widetilde{p}\in\widetilde{\Gamma}}R_{\widetilde{p}}\cdot\widetilde{p}$ with 
\begin{equation*}
    R_{\widetilde{p}}(\pi):= d_{\widetilde{p}}(\pi)\big(2-2\cdot h(\pi(\widetilde{p}))\big)-\big(2-2\cdot \widetilde{h}(\widetilde{p})\big) -\sum_{\widetilde{v}\in T_{\widetilde{p}}\Gamma}\big(d_{\widetilde{v}}(\pi)-1\big),
\end{equation*}
 so the morphism is unramified if $R_{\widetilde{p}}(\pi)=0$ for all $\widetilde{p}\in\widetilde{\Gamma}$. From now on, we refer to an unramified harmonic morphism of degree $2$ as an \emph{unramified double cover}. 

\begin{remark}
The condition in Definition \ref{def_harmonicmorphism} says that $f$ pulls back harmonic functions on $\Gamma$ to harmonic functions on $\widetilde{\Gamma}$ (see \cite{MikhalkinZharkov,ABBRI} for more on this point of view). 
\end{remark}

Let $X$ be a smooth projective curve over a non-Archimedean field $K$. The non-Archimedean skeleton $\Gamma_X$ of the Berkovich space $X^{an}$ in the sense of \cite{Berkovich_book} naturally has the structure of a tropical curve. We write $\rho_X$ for the natural strong deformation retraction of $X^{an}$ onto $\Gamma_X$  and refer the interested reader to \cite[Section 4]{Berkovich_book} and \cite{BakerPayneRabinoff_nonArchcurves} for details on this construction.

Let $\pi\colon \widetilde{X}\rightarrow X$ be a finite morphism of smooth projective curves over $K$. By \cite[Theorem A]{ABBRI}, the restriction of $\pi^{an}\colon \widetilde{X}^{an}\rightarrow X^{an}$ to $\Gamma_{\widetilde{X}}\subseteq \widetilde{X}^{an}$ defines a finite harmonic morphism $\pi^{trop}\colon \Gamma_{\widetilde{X}}\rightarrow\Gamma_X$. If $\pi$ has degree $d$, so does $\pi^{trop}$, and if $\pi$ is unramified, so is $\pi^{trop}$. Alternatively, we may apply the valuative criterion of properness to the moduli space of admissible covers to find a simultaneous semistable reduction of $\pi\colon \widetilde{X}\rightarrow X$ over a finite extension of $K$. Then $\pi^{trop}$ is precisely the induced map on dual tropical curves (see \cite{CavalieriMarkwigRanganathan_admissiblecovers} for details). 

\subsection{Picard groups, Jacobian, and the Abel-Jacobi map}

Let $\Gamma$ be a tropical curve. A \emph{divisor} $D$ on $\Gamma$ is a finite formal sum $\sum_{i+1}^n a_i p_i$ over points $p_i\in\Gamma$ (with $a_i\in\Z$). We write $\deg D=\sum_{i=1}^n a_i$ for the degree of a divisor $D=\sum_{i=1}^n a_ip_i$ and $\Div_d(\Gamma)$ for the  divisors of degree $d$ on $\Gamma$. A \emph{rational function} on $\Gamma$ is a continuous piecewise-linear function $f\colon\Gamma\rightarrow\R$ with integer slopes. Write $\Rat(\Gamma)$ for the abelian group of rational functions on $\Gamma$. There is a homomorphism
\begin{equation*}\begin{split}
\div\colon \Rat(\Gamma)&\longrightarrow \Div_0(\Gamma),\\
f&\longmapsto \sum_{p\in\Gamma}\ord_p(f) \cdot p,
\end{split}\end{equation*}
where $\ord_p(f)$ denotes the sum of the outgoing slopes of $f$ at $p$. 
Divisors in the image of $\div$ are referred to as \emph{principal divisors} and denoted by $\PDiv(\Gamma)$. 
The \emph{Picard group} of $\Gamma$ is defined to be the quotient \begin{equation*}
  \Pic(\Gamma)=\Div(\Gamma)/PDiv(\Gamma) \ .
\end{equation*}
The degree function descends to $\Pic(\Gamma)$ and every $\Pic_d(\Gamma)$ is naturally a torsor over $\Pic_0(\Gamma)$.

 Choose a model $(G,\ell)$ of $\Gamma$. Let $M=H_1(\Gamma,\Z)$ and $M'=H^1(\Gamma,\Z)$, and consider the edge length pairing 
\begin{equation}\begin{split}\label{eq_edgelengthpairing}
    \langle .,.\rangle \colon M'\times M&\longrightarrow \R\\
    \Big(\sum_{e\in E(G)}a_e [e^\ast], \sum_{e\in E(G)}b_e [e]\Big)&\longmapsto \sum_{e\in E(G)} a_eb_e\ell(e) \ .
\end{split}\end{equation}
Set $N_\R=\Hom(M,\R)$ and $N_\R'=\Hom(M',\R)$. By the universal coefficient theorem, the natural map $\lambda_{\Theta}\colon M'\rightarrow M$ induced by 
\begin{equation*}
    \sum_{e\in E(G)}a_e[e]\longmapsto \sum_{e\in E(G)}a_e [e^\ast]
\end{equation*}
is an isomorphism and therefore defines a principal polarization.

\begin{definition}
The principally polarized tropical abelian variety
$\Alb(\Gamma)=N_\R/M'$ is called the \emph{Albanese variety} of $\Gamma$ and its dual $\Jac(\Gamma)=N_\R'/M$  the \emph{Jacobian variety} associated to $\Gamma$. 
\end{definition}

Notice that the definition of both $\Jac(\Gamma)$ and $\Alb(\Gamma)$ does not depend on the choice of the model $(G,\ell)$ and that, thanks to the principal polarization, the Jacobian and the Albanese torus are naturally isomorphic.  

A \emph{$1$-form} on $\Gamma$ is a formal $\R$-linear sum over elements $de$, as $e$ ranges over the edges of $G$, subject to the condition that $de=-de'$ whenever $e$ and $e'$ represent the same edge in opposite directions. 
A $1$-form $\omega=\sum\omega_e de$ is \emph{harmonic} if for every vertex $v$ of $G$, the sum $\sum \omega_e$ over all outgoing edges $e$ at $v$ is zero. The space $\Omega(\Gamma)$ of $1$-forms  is a real vector space that is naturally isomorphic to $H^1(\Gamma,\R)$ by \cite[Lemma 2.1]{BakerFaber_tropicalAbelJacobi}. Using this observation we have a natural isomorphism
\begin{equation*}
    \Jac(\Gamma)\simeq\Omega(\Gamma)^\ast/H_1(\Gamma) \,
\end{equation*}
where $\Omega(\Gamma)^\ast$ denotes the $\R$-linear dual of $\Omega(\Gamma)$ and where we send a cycle $[\gamma]$ in $H_1(\Gamma)$ to the $\R$-linear homomorphism 
\begin{equation*}
    \omega \longmapsto \int_\gamma \omega
\end{equation*}
in $\Omega(\Gamma)^\ast$. 

Fix a point $q$ in $\Gamma$. There is a tropical Abel-Jacobi map $\alpha_q\colon \Gamma\longrightarrow \Jac(\Gamma)$ constructed as follows: Let $p\in\Gamma$. Fix a path $\gamma$ connecting $q$ to $p$ and consider the homomorphism $\widetilde{\alpha}_{q}(p,\gamma)\in\Omega(X)^\ast$ given by 
\begin{equation*}
    \omega\longmapsto \int_\gamma\omega \ .
\end{equation*}
 If we had chosen a different path $\gamma'$ between $q$ and $p$, then the difference $\int_\gamma \omega -\int_{\gamma'}\omega$ is an element in $M=H_1(\Gamma)$ and the association $p\mapsto \widetilde{\alpha}_q(p,\gamma)$ descends to the \emph{Abel-Jacobi map}
\begin{equation*}
    \alpha_q\colon\Gamma\longrightarrow\Jac(\Gamma)=\Omega(\Gamma)^\ast/H_1(\Gamma) \ .
\end{equation*}
We refer the reader to \cite{MikhalkinZharkov, BakerFaber_tropicalAbelJacobi} for details on this construction.

The Abel-Jacobi map extends linearly to a homomorphism
$\alpha_{q,\ast}\colon\Div_0(\Gamma)\rightarrow \Jac(\Gamma)$ and the tropical analogue of the Abel-Jacobi-Theorem \cite[Theorem 6.2]{MikhalkinZharkov} states that $\alpha_{q,\ast}$ descends to a canonical isomorphism  $\Pic_0(\Gamma)\xrightarrow{\sim}\Jac(\Gamma)$. Under this isomorphism, the tropical Abel-Jacobi map is given by $p\mapsto[p-q]$ (also see \cite[Theorem 3.4]{BakerFaber_tropicalAbelJacobi}).

\subsection{A tropical norm map}\label{section_tropicalnormmapI}

In this section we introduce a tropical analogue of the norm map. We proceed in complete analogy with the algebraic situation, as e.g. explained in \cite[Appendix B.1]{ACGHI}. Let $\pi\colon\widetilde{\Gamma}\rightarrow\Gamma$ be a finite harmonic morphism. Choose a model for $\Gamma$ and $\widetilde{\Gamma}$ such that $\pi$ is linear on every edge.

\begin{lemma}\label{lem:slopes}
There is a unique $\R$-linear homomorphism 
\begin{equation*}
\Nm_\pi\colon \Rat(\widetilde{\Gamma})\longrightarrow \Rat(\Gamma)
\end{equation*}
such that 
\begin{enumerate}
    \item For $\widetilde{f}\in\Rat(\widetilde{\Gamma})$  the slope of $\Nm_\pi(\widetilde{f})$ at every edge $e$ of $G$ equals the sum of the slopes of $\widetilde{f}$ at all edges $\widetilde{e}$ of $\widetilde{G}$ that map to $e$, and
    \item $\Nm_\pi(1)=d$, where $1$ is the rational function whose values are constantly $1$.
\end{enumerate}

\end{lemma}
We refer to $\Nm_\pi(\widetilde{f})$ as the \emph{norm} of $\widetilde{f}$.

\begin{proof}[Proof of Lemma \ref{lem:slopes}]
For $\widetilde{f}\in\Rat(\widetilde{\Gamma})$, define a function $f$ on $\Gamma$ by
\begin{equation*}
f(p) = \sum_{\substack{\widetilde{p}\in\widetilde{\Gamma}\\\pi(\widetilde{p})=p}} d_{\widetilde{p}}(\pi)\cdot\widetilde{f}(\widetilde{p}).
\end{equation*}
We have $f\in\Rat(\Gamma)$ since $\pi$ is a finite harmonic morphism. We set $\Nm_{\pi}(\widetilde{f})=f$ and note that $\Nm_\pi(\widetilde{f}+\widetilde{g})=\Nm_\pi(\widetilde{f})+\Nm_\pi(\widetilde{g})$ as well as $\Nm_\pi(c\cdot \widetilde{f})=c\cdot\Nm_\pi(\widetilde{f})$ and $\Nm_\pi(1)=d$. 


Let $e$ be an edge of $\Gamma$ connecting vertices $v$ and $w$. Let $\{\widetilde{e}\}$ be the collection of preimages of $e$ in $\tilde{\Gamma}$,  and denote $v_{\widetilde{e}}$ and $w_{\widetilde{e}}$ the endpoints of each  $\tilde{e}$  (where $\pi(v_{\widetilde{e}})=v, \pi(w_{\widetilde{e}})=w$). We need to show that $\sum_{\widetilde{e}}\text{slope}_{\widetilde{f}}(\widetilde{e}) =  \text{slope}_{f}(e)$.
Using $d_{\widetilde{e}}(\pi)\cdot \ell(\widetilde{e})=\ell(e)$, we find:
\begin{equation*}\begin{split}
    \sum_{\widetilde{e}}\text{slope}_{\widetilde{f}}(\widetilde{e}) &= \sum_{\widetilde{e}}\frac{\widetilde{f}(w_{\widetilde{e}})-\widetilde{f}(v_{\widetilde{e}})}{\ell(\widetilde{e})} \\
    &=\frac{1}{\ell(e)}\cdot \sum_{\widetilde{e}}d_{\widetilde{e}}(\pi)\big(\widetilde{f}(w_{\widetilde{e}})-\widetilde{f}(v_{\widetilde{e}})\big)  \\
    &=\frac{1}{\ell(e)}\cdot \Big( \sum_{w_{\widetilde{e}}} d_{w_{\widetilde{e}}}(\pi)\cdot \widetilde{f}(w_{\widetilde{e}}) - \sum_{v_{\widetilde{e}}} d_{w_{\widetilde{e}}}(\pi)\cdot \widetilde{f}(v_{\widetilde{e}}) \Big)\\
&=\frac{ f(w) - f(v)}{\ell(e)} = \text{slope}_{f}(e) \ .\\
\end{split}\end{equation*}
\end{proof}

The following Proposition \ref{prop_divnorm=normdiv} shows that the natural pushforward map $\Nm_\pi=\pi_\ast \colon\Div(\widetilde{\Gamma})\rightarrow\Div(\Gamma)$ 
is compatible with the norm map. 

\begin{proposition}\label{prop_divnorm=normdiv}
For $\widetilde{f}\in\Rat(\widetilde{\Gamma})$ we have:
\begin{equation*}
    \pi_\ast \big(\div(\widetilde{f})\big) =\div \big(\Nm_\pi(\widetilde{f})\big) \ .
\end{equation*}
\end{proposition}

\begin{proof}
Let $\widetilde{f}\in\Rat(\widetilde{\Gamma})$, and denote $f = \Nm_\pi(\widetilde{f})$.  Let $p$ be a point of $\Gamma$. 
From Lemma \ref{lem:slopes}, it follows that the slope of $\psi$ at every edge $e$ emanating from $p$ equals the  sum of the slopes at edges of $\widetilde{\Gamma}$ mapping down to $e$. From this we see that  
\begin{equation*}
\div(f)(p) = \sum_e \text{slope}_{f}(e) 
 =\sum_{\pi(\tilde{p})=p}  \sum_{\pi(\tilde{e})=e} \text{slope}_{\widetilde{f}}(\tilde{e}) 
=\sum_{\pi(\tilde{p})=p} (\div\widetilde{f})(\tilde{p}) = \pi_*\div(\widetilde{f}),
\end{equation*}
as claimed. 
\end{proof}

Proposition \ref{prop_divnorm=normdiv} implies that the pushforward map naturally descends to a homomorphism 
\begin{equation*}
    \Nm_\pi\colon \Pic(\widetilde{\Gamma})\longrightarrow\Pic(\Gamma) \ . 
\end{equation*}
which we call the \emph{norm map} associated to $\pi\colon\widetilde{\Gamma}\rightarrow\Gamma$. The norm respects degrees and we indiscriminately write $\Nm_\pi$ for the induced map $\Pic_d(\widetilde{\Gamma})\rightarrow\Pic_d(\Gamma)$ for all $d\geq 0$. 

We will now show that $\Nm_\pi\colon\Pic_0(\widetilde{\Gamma})\rightarrow\Pic_0(\Gamma)$  is an integral homomorphism of principally polarized tropical abelian varieties. Recall from \cite{BakerFaber_tropicalAbelJacobi} that there is a natural \emph{pullback homomorphism} $\pi^\ast\colon\Alb(\Gamma)\rightarrow \Alb(\widetilde{\Gamma})$ that is induced by the pullback of harmonic forms along $\pi$.

\begin{proposition}\label{prop_NormAbelJacobi}
Fix $\widetilde{q}\in\widetilde{\Gamma}$ and write $q=\pi(\widetilde{q})$. Then the diagram 
of Abel-Jacobi maps
\begin{equation}\label{eq_AbelJacobi=Norm}\begin{tikzcd}
    \widetilde{\Gamma}\arrow[d,"\pi"'] \arrow[rr,"\alpha_{\widetilde{q}}"]& & \Jac(\widetilde{\Gamma})\arrow[d,"(\pi^\ast)'"] \\
    \Gamma \arrow[rr,"\alpha_q"]&& \Jac(\Gamma)
\end{tikzcd}\end{equation}
is commutative.
\end{proposition}

For a finite morphism of algebraic curves the norm map is dual to the pullback homomorphism  (see e.g. \cite[Section 1]{Mumford_Prym}). The following Corollary \ref{cor_norm=dualtopullback} provides us with a tropical analogue of this observation.

\begin{corollary}\label{cor_norm=dualtopullback}
The tropical norm map $\Nm_\pi\colon\Jac(\widetilde{\Gamma})\rightarrow\Jac(\Gamma)$ is dual to the pullback homomorphism $\pi^\ast\colon\Alb(\Gamma)\rightarrow \Alb(\widetilde{\Gamma})$.
\end{corollary}

Corollary \ref{cor_norm=dualtopullback} in particular shows that the norm map is an integral homomorphism of tropical abelian varieties. 

\begin{proof}[Proof of Proposition \ref{prop_NormAbelJacobi}] Let $\widetilde{p}\in\widetilde{\Gamma}$ and set $p=\pi(\widetilde{p})$. Choose models $\widetilde{G}$ and $G$ of both $\widetilde{\Gamma}$ and $\Gamma$ such that both $\widetilde{p}$ and $\widetilde{q}$ are vertices of $\widetilde{G}$ and $\pi$ is cellular. Let $\widetilde{\gamma}$ be a path connecting $\widetilde{q}$ to $\widetilde{p}$ and write 
\begin{equation*}
    [\widetilde{\gamma}] = \sum_{\widetilde{e}\in E (\widetilde{G})} c_{\widetilde{e}}[\widetilde{e}]
\end{equation*}
for its associated $1$-chain. Recall that $\pi^\ast\colon  \Omega(\Gamma)\rightarrow\Omega(\widetilde{\Gamma})$ is given by 
\begin{equation*}
    \sum_{e\in E(G)}\omega_e de \longmapsto \sum_{e\in E(G)}\sum_{\substack{\widetilde{e}\in E(\widetilde{G})\\ \pi(\widetilde{e})\subseteq e}}d_{\widetilde{e}}(f)\cdot \omega_e de \ .
\end{equation*}
For all $\omega=\sum_{e} \omega_ede\in\Omega(\Gamma)$ we verify 
\begin{equation*}
\begin{split}
    \int_{\widetilde{\gamma}} \pi^\ast(\omega) &=\sum_{e\in E(G)}\omega_e\cdot \Big(\sum_{\substack{\widetilde{e}\in E(\widetilde{G})\\ \pi(\widetilde{e})\subseteq e}} b_{\widetilde{e}} \cdot d_{\widetilde{e}}(f)\cdot \ell(\widetilde{e})\Big)\\
   &= \sum_{e\in E(G)}\omega_e \cdot \Big(\sum_{\substack{\widetilde{e}\in E(\widetilde{G})\\ \pi(\widetilde{e})\subseteq e}} b_{\widetilde{e}} \Big)\cdot \ell(e)= \int_{\pi_\ast[\widetilde{\gamma}]} \omega \ ,
\end{split}
\end{equation*}
using $d_{\widetilde{e}}\cdot \ell(\widetilde{e})=\ell(e)$ and this shows the commutativity of diagram \eqref{eq_AbelJacobi=Norm}. 
\end{proof}

\subsection{Tropical Prym varieties}
In this section we recall the definition of the tropical Prym variety $\Pr(\Gamma,\pi)$ from \cite[Section 6]{JensenLen_thetachars} and study its basic properties. We refer the reader to \cite[Appendix C]{ACGHI} for the classical analogue of this story. Fix an unramified double cover $\pi\colon\widetilde{\Gamma}\rightarrow\Gamma$. 

\begin{definition}
The \emph{dilation} cycle associated with the unramified double cover $\pi$ is the collection of points $p\in\Gamma$ whose preimage in $\widetilde\Gamma$ consists of a single point $\widetilde{p}$ with $d_{\widetilde{p}}(\pi)= 2$.
\end{definition}

The following Lemma generalizes \cite[Corollary 5.5]{JensenLen_thetachars}.

\begin{lemma}\label{lem:dilationCycle}
The dilation cycle is a union of cycles and isolated points. 
\end{lemma}
\begin{proof}
Let $\widetilde{p}\in\widetilde\Gamma$ be such that $p=\pi(\widetilde{p})$ is in the dilation cycle, namely  $d_{\widetilde{p}} = 2$. 
Since $\pi$ is unramified, we find that 
\[
2h(\widetilde{p})-2=2\cdot (2h(p)-2)+\ell\ ,
\]
where $\ell$ denotes the number of edges emanating from $p$ that are in the dilation cycle.
Since both $h(p)$ and $h(\widetilde{p})$ are integers, $\ell$ must be an even number. Whenever $\ell>0$, the point $p$ is part of a cycle, and otherwise  $p$ is an isolated point. 
\end{proof}

Note that isolated points may only occur  where the weight is positive.

\begin{example}\label{ex:unramified_cover}
Suppose that $\widetilde{\Gamma}$ consists of a single vertex with weight $1$ and two loops of length $1$, and $\Gamma$ consists of a single vertex of weight $0$, and two loops of length $2$. Then the map $\pi:\widetilde\Gamma\to\Gamma$ which sends vertex to vertex and dilates the edges by a factor of $2$ is an unramified double cover. In this case, the dilation cycle is the entire graph.

Now, assume that $\Gamma$ consists of a vertex $p$ of weight $1$ and a loop connected by an edge. Assume that $\widetilde{\Gamma}$ consists of a point $\widetilde{p}$ of weight one connected by edges to two loops. Then the dilation cycle consists only of the weighted point $p$.
\end{example}

\begin{lemma}\label{lemma_spanningtrees}

Let $\pi\colon\widetilde{\Gamma}\rightarrow\Gamma$ be an unramified double cover. There are bases 
\begin{equation*}
    \epsilon_{1},\ldots, \epsilon_{a},\epsilon_{a+1},\ldots,\epsilon_{a+b},\epsilon_{a+b+1},\ldots, \epsilon_{a+b+c}
\end{equation*}
of $H_1(\Gamma)$ and
\begin{equation*}
    \widetilde{\epsilon}^{\pm}_{1},\ldots, \widetilde{\epsilon}^{\pm}_{a},\widetilde{\epsilon}_{a+1},\ldots,\widetilde{\epsilon}_{a+b}, \widetilde{\epsilon}_{a+b+1},\ldots, \widetilde{\epsilon}_{a+b+c}, \ldots, \widetilde{\epsilon}_{a+b+c+1},\ldots, \widetilde{\epsilon}_{a+b+c+d}
\end{equation*}
of $H_1(\widetilde{\Gamma})$ such that 
\begin{equation*}
    \left\{\begin{array}{ll}
        \pi_\ast(\widetilde{\epsilon}_i^{\pm})=\epsilon_i & \textrm{ for } i=1,\ldots, a\\
        \pi_\ast(\widetilde{\epsilon}_i)=2\epsilon_i & \textrm{ for } i=a+1,\ldots, a+b\\
        \pi_\ast(\widetilde{\epsilon}_i)=\epsilon_i & \textrm{ for } i=a+b+1,\ldots, a+b+c\\
        \pi_\ast(\widetilde{\epsilon}_i)=0 & \textrm{ for } i=a+b+1,\ldots, a+b+c+d 
    \end{array}\right .
\end{equation*}





\end{lemma}

\begin{proof} 

Let us first consider the case that the dilation cycle is empty, so that $\pi$ is a topological cover. Fix an orientation for $\Gamma$ and a corresponding orientation for $\widetilde\Gamma$. Let $T$ be any spanning tree  of $\Gamma$. Then $\widetilde{\Gamma}$ may be described as follows: take two copies $T_1$ and $T_2$ of $T$ and, for every edge $e$ in the complement of $T$ we have two choices for lifts in the complement of $T_1\cup T_2$.
We want to impose that $\widetilde{\Gamma}$ is connected. So we need at least one connection $\widetilde{e}_g^+$ between $T_1$ and $T_2$. A spanning tree for $\widetilde{\Gamma}$ is given by $\widetilde{T}=T_1\cup T_2\cup \widetilde{e}_g^+$. We  write $\widetilde{e}_1^\pm, \widetilde{e}_2^\pm, \ldots, \widetilde{e}_{g-1}^\pm,\widetilde{e}_g^-$ for the edges in the complement of $\widetilde{T}$, such that  $e_i=\pi(\widetilde{e}_i^\pm)$ for each $i$. 
For each $i$, let $\widetilde\delta_i^\pm$ be the smallest cycle supported on $\widetilde T\cup  \widetilde e_i^{\pm}$. The cycles $\widetilde\delta_1^\pm,\ldots\widetilde\delta_{g-1}^\pm,\widetilde\delta_g^-$ form a basis for $H_1(\widetilde\Gamma)$ since they were obtained from the complement of a spanning tree. The desired basis is now given by taking 
\[
\widetilde\epsilon_i^+ = \widetilde\delta_i^+
\]
and
\[
\widetilde\epsilon_i^- = -\widetilde\delta_i^+ + \widetilde\delta_g^-
\]
for $i=1,\ldots, g-1$, as well as 
\begin{equation*}
    \epsilon_g=\widetilde{e}_g^{-}.
\end{equation*}
One easily checks that $\pi_*(\epsilon_i^+) = \pi_*(\epsilon_i^-)$ for $i=1,2,\ldots, g-1$, and  $\pi_\ast (\widetilde{\epsilon}_g)=2\epsilon_g$. So in this case we have $a=g-1$ and $b=1$. 

Note that, if we had considered a disconnected topological double cover, the same argument  would obtain bases $\epsilon_1,\ldots, \epsilon_g$ of $H_1(\Gamma)$ and $\widetilde{\epsilon}^{\pm}_1,\ldots, \widetilde{\epsilon}_g^{\pm}$ of $H_1(\widetilde{\Gamma})$ with $\pi_\ast(\widetilde{\epsilon}_i^{\pm})=\epsilon_i$ for all $i=1,\ldots, g$, so in that case $a=g$ and $b=0$.

Suppose now that the dilation cycle $\gamma$ in non-empty. Then a basis $\epsilon_i$ of cycles in $\gamma$ gives rise to a basis $\widetilde{\epsilon}_i$ of its preimage $\widetilde{\gamma}=\pi^{-1}(\gamma)$ in $\widetilde{\Gamma}$ with $\pi_\ast(\widetilde{\epsilon}_i)=\epsilon_i$.
Let $\Gamma_1,\ldots,\Gamma_\ell$ be the connected components of $\Gamma\setminus\gamma$ and  $\widetilde\Gamma_1,\ldots,\widetilde\Gamma_\ell$  their preimages in $\widetilde\Gamma$. Write $\widetilde{\Gamma}_i=\widetilde{\Gamma}_i^+\cup\widetilde{\Gamma}_i^-$ for the two components of $\widetilde{\Gamma}_i$ whenever it is disconnected. By the first part of the proof we find basis vectors for each of the covers $\widetilde{\Gamma}_i\rightarrow \Gamma_i$.

In order to describe the additional cycles formed by attaching all the different components,  let us consider the tropical curve $\Gamma_0$ obtained by contracting each $\Gamma_i$ and each connected component of $\gamma$ to a point respectively (keeping edges that connect to the rest of the graph). Denote by $\widetilde{\Gamma}_0$ the tropical curve obtained by simultaneously contracting all components of the $\widetilde{\Gamma_i}$ and of $\pi^{-1}(\gamma)$. 
Note that the minimal underlying graphs of $\Gamma_0$ and $\widetilde{\Gamma}_0$ are bipartite: an edge in $\Gamma_0$  always connects a vertex $v_i$ arising from some component $\Gamma_i$ to a vertex $w$ arising from a component of  $\gamma$. Similarly in $\widetilde{\Gamma}_0$, edges connect vertices arising from some $\widetilde{\Gamma}_i$  to vertices arising  from $\widetilde{\gamma}$. The morphism $\pi$ induces an unramified harmonic morphism $\pi_0\colon \widetilde{\Gamma}_0\rightarrow \Gamma_0$. 

Now consider a vertex $v_i$ in $\Gamma_0$ originating from a component $\Gamma_i$. The star of $v_i$, denoted $S(v_i)$, is a bouquet of banana graphs. 
When the preimage of $v_i$ in $\widetilde{\Gamma}_0$ consists of a single vertex $\widetilde{v}_i$, then $S(\widetilde v_i)$ is obtained by doubling each edge of $S(v_i)$. When the preimage of $v_i$ consists of two vertices $\widetilde{v}_i^\pm$, the preimage of $S(v_i)$  consists of two copies of $S(v_i)$ glued at their ends. Either way, denoting $r_i$ and $\widetilde r_i$ the first Betti numbers of $S(v_i)$ and $S(\widetilde v_i)$ it is straightforward to construct bases $\theta_1,\theta_2,\ldots,\theta_{r_i}$ and $\widetilde\theta_1,\widetilde\theta_2,\ldots,\widetilde\theta_{r_i}$ of their homologies, such that 
\begin{equation*}
    \pi_{0,\ast}(\widetilde{\theta}_k)=\begin{cases}
                      \theta_i & \textrm{ if } 1\leq k\leq r_i\\
                      0 &\textrm{ if } r_i+1\leq k\leq \widetilde r_i \ .
    \end{cases}
\end{equation*}
These cycles lift to cycles $\epsilon_1,\epsilon_2,\ldots,\epsilon_{r_i}$ and  $\widetilde\epsilon_1,\widetilde\epsilon_2,\ldots,\widetilde\epsilon_{\widetilde r_i}$ such that 
\begin{equation*}
    \pi_{0,\ast}(\widetilde{\epsilon}_k)=\begin{cases}
                      \epsilon_i & \textrm{ if } 1\leq k\leq r_i\\
                      0 &\textrm{ if } r_i+1\leq k\leq \widetilde r_i \ .
    \end{cases}
\end{equation*}



Finally, consider the graphs $\Sigma$ and $\widetilde\Sigma$ obtained from $\Gamma_0$ and $\widetilde\Gamma_0$ by identifying all the edges in each banana graph to a single edge. Every simple cycle  $\theta$ in $\Sigma$ may be lifted to a simple cycle $\widetilde{\theta}$ in $\widetilde\Sigma$. Now, $\theta$ and $\widetilde{\theta}$ lift to cycles $\widetilde\epsilon$ and $\epsilon$  such that $\pi_\ast(\widetilde\epsilon) = \epsilon$.

Combining all of these choices we find bases of $H_1(\Gamma)$ and of $H_1(\widetilde{\Gamma})$ with the desired properties. 

\end{proof}

\begin{definition}
Let $\pi:\widetilde\Gamma\to\Gamma$ be an unramified double cover.   The \emph{Prym variety} $\Pr(\Gamma,\pi)$ associated to $\pi$ is the connected component of $\ker(\Nm_\pi)$ containing $0$. 
\end{definition}

By Proposition \ref{prop_tropker&coker} the Prym variety naturally has the structure of tropical abelian variety (with the polarization induced from $\Pic_0(\widetilde{\Gamma})$). 

\begin{remark}
Our definition differs from the one given in \cite{JensenLen_thetachars} when the metric graphs are augmented, since we do not add virtual loops in place of vertex-weights. As a result, the dimension of the Prym variety in our case is smaller when there are non-trivial weights. 
\end{remark}


 Let $\pi\colon \widetilde{X}\rightarrow X$ be an unramifed double cover. Then $\ker(\Nm_\pi)$ has two components and, by \cite{Mumford_Prym}, there is a principal polarization $\Xi$ on $\Pr(X,\pi)$ such that 
\begin{equation*}
    i^\ast \widetilde{\Theta} = 2\cdot \Xi \ ,
\end{equation*}
where $i^\ast\widetilde{\Theta}$ denotes the induced polarization from the Theta-divisor on $\Jac(\widetilde{X})$. The following Theorem \ref{thm_Prymvariety} is  a tropical analogue of this result and expands on  \cite[Proposition 6.1]{JensenLen_thetachars}. 

Let $\pi:\widetilde\Gamma\to\Gamma$ be an unramified double cover. Denote by $\Gamma^0$ and $\widetilde{\Gamma}^0$ the un-augmented graphs obtained from $\Gamma$ and $\widetilde{\Gamma}$ respectively by forgetting the vertex weights. Let $g_0$ and $h_0$ be the genera of $\Gamma^0$ and $\widetilde\Gamma^0$ (i.e. their first Betti numbers).

\begin{theorem}\label{thm_Prymvariety}
Let $\pi:\widetilde\Gamma\to\Gamma$ be an unramified double cover. Then $\ker(\Nm_\pi)$ is a union of real tori of dimension $h_0-g_0$ in $\Jac(\Gamma)$. \begin{enumerate}[(i)]
\item When the dilation cycle of $\pi$ is trivial then $\ker(\Nm_\pi)$ has two connected components. Moreover, there is a principal polarization $\lambda_{\Xi}$ on $\Pr(\Gamma,\pi)$ such that 
    \begin{equation}\label{eq_2times}
        i^\ast\lambda_{\Theta}=2\cdot\lambda_\Xi \ .
    \end{equation}
\item When the dilation cycle of $\pi$ is non-trivial, then $\ker(\Nm_\pi)$ is connected and there is a natural principal polarization $\lambda_\Xi$ on $\Pr(\Gamma,\pi)$.
\end{enumerate}
\end{theorem}

We will see from the proof that formula \eqref{eq_2times} does not hold when the dilation cycle is non-trivial.

\begin{proof}
The claim about the connected components is proved almost verbatim as in \cite[Proposition 6.1]{JensenLen_thetachars} and is left to the avid reader. We focus on the statement about the polarizations.

Choose bases of $H_1(\Gamma)$ and  $H_1(\widetilde{\Gamma})$ as in Lemma \ref{lemma_spanningtrees}.  
The kernel of $\phi_\ast\colon H_1(\widetilde{\Gamma})\rightarrow H_1(\Gamma)$ is spanned by the vectors 
\begin{equation*}
\widetilde{\epsilon}_1^+-\widetilde{\epsilon}_1^-,\ldots, \widetilde{\epsilon}_a^+-\widetilde{\epsilon}_a^-, \widetilde{\epsilon}_{a+b+c+1},\ldots, \widetilde{\epsilon}_{a+b+c+d}\ .
\end{equation*}
and the image of $\phi^\ast\colon H^1(\Gamma)\rightarrow H^1(\widetilde{\Gamma})$ is generated by 
\begin{equation*}
    (\widetilde{\epsilon}_1^+)^\ast+(\widetilde{\epsilon}_1^-)^\ast,\ldots, (\widetilde{\epsilon}_a^+)^\ast+(\widetilde{\epsilon}_a^-)^\ast, 2\cdot\widetilde{\epsilon}_{a+1}^\ast,\ldots,2\cdot \widetilde{\epsilon}_{a+b}^\ast,\widetilde{\epsilon}_{a+b+1}^\ast,\ldots, \widetilde{\epsilon}_{a+b+c}^\ast\ .
\end{equation*}
The induced polarization is the one given by sending $\widetilde{\epsilon}_i^+-\widetilde{\epsilon}_i^-$ to $(\widetilde{\epsilon}_i^+)^\ast-(\widetilde{\epsilon}_i^-)^\ast$ for $i=1,\ldots, a$ as well as by sending $\widetilde{\epsilon}_{j}$ to $\widetilde{\epsilon}_{j}^\ast$ for $j=a+b+c+1,\ldots, a+b+c+d$. This is not an isomorphism, unless $a=0$, because in the cokernel, $(\widetilde{\epsilon}_i^+)^\ast-(\widetilde{\epsilon}_i^-)^\ast$ is identified with $2\cdot\widetilde{\epsilon}_i^\ast$. We, however, may define a principal polarization $\lambda_\Xi$ by sending $\widetilde{\epsilon}_i^+-\widetilde{\epsilon}_i^-$ to $\widetilde{\epsilon}_i^\ast$ for $i=1,\ldots, a$ and by sending $\widetilde{\epsilon}_{j}$ to $\widetilde{\epsilon}_{j}^\ast$ for $j=a+b++c+1,\ldots, a+b+c+d$. \end{proof}

\begin{example}
If $\pi:\widetilde\Gamma\to\Gamma$ is the first map in Example \ref{ex:unramified_cover}, then  the corresponding Prym is trivial.   
\end{example}

\begin{example}
Consider a double cover as indicated in Figure \ref{fig_smileymouth}. We have $a=0$. The dilation cycle is given by the two loops $e_1$ and $e_2$ and we find that $b=2$ with the two vectors $\widetilde{\epsilon}_1$ and $\widetilde{\epsilon}_2$ associated to $\widetilde{e}_1$ and $\widetilde{e}_2$. We have $\pi_\ast \widetilde{\epsilon}_i=\epsilon_i$ for $i=1,2$. The last basis vector $\widetilde{\epsilon}_3$ is the one associated to the $\widetilde{e}_3$ and we have $\pi_\ast \epsilon_3=0$. The kernel of $\Nm_\pi$ has only one component and the polarization on $\Pr(\Gamma,\pi)$ induced from $\Jac(\widetilde{\Gamma})$ is already principal. 
\end{example}

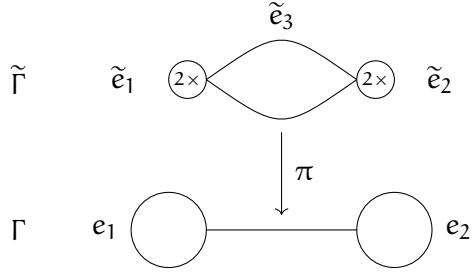
\begin{figure}
    \centering
        \begin{tikzpicture}
            \draw (0,0) circle (0.5);
            \draw (3,0) circle (0.5);
            \draw (0.5,0) -- (2.5,0);
            \node at (-0.85,0) {$e_1$};
            \node at (3.85,0) {$e_2$};

            \draw (0.25,2) circle (0.25);
            \draw (2.75,2) circle (0.25);
            \draw (0.5, 2) .. controls (1.5,1.3) .. (2.5, 2);
            \draw (0.5, 2) .. controls (1.5,2.7) .. (2.5, 2);
            \node at (-0.6,2) {$\widetilde{e}_1$};
            \node at (3.6,2) {$\widetilde{e}_2$};
            \node at (1.5,2.85) {$\widetilde{e}_3$};
            \node at (0.25,2) {\tiny $2\times$};
            \node at (2.75,2) {\tiny $2\times$};

            \draw[->]  (1.5,1.3) -- (1.5,0.2);
            \node at (1.8,0.75) {$\pi$};
            
            \node at (-2,0) {$\Gamma$};
            \node at (-2,2) {$\widetilde{\Gamma}$};

        \end{tikzpicture}
    \caption{A double cover whose dilation cycle is $e_1\cup e_2$.}
    \label{fig_smileymouth}
\end{figure}


\section{The skeleton of $\Pr(X,\pi)$ -- Proof of Theorem \ref{thm_skeletonofPrym=Prymofskeleton}} \label{section_skeletonofPrym}

In this section we first recall the theory of non-Archimedean uniformization of abelian varieties as developed in \cite{Bosch_goodreduction, BoschLuetkebohmert_uniformizationII, BoschLuetkebohmert_degeneratingabelvar} and describe their non-Archimedean skeletons \cite[Section 6.5]{Berkovich_book} as polarized tropical abelian varieties; we  mostly follow \cite[Section 4.1--4.3]{BakerRabinoff_skelJac=Jacskel}, \cite{FosterRabinoffShokriehSoto_tropicaltheta}, and the beautiful monograph \cite{Luetkebohmert_curves&Jacobians}. We then deduce Theorem \ref{thm_skeletonofPrym=Prymofskeleton} from the functoriality of this framework and Theorem \ref{thm_Prymvariety} above.



\subsection{Non-Archimedean uniformization}
Let $A$ be an abelian variety over $K$ with split semi-abelian reduction. The universal cover $E^{an}$ of $A^{an}$ carries a unique structure of a $K$-analytic group such that the covering map $E^{an}\rightarrow A^{an}$ is a morphism $K$-analytic groups. The $K$-analytic group is the analytification of an algebraic group $E$ that arises as the extension of an abelian variety $B$ by a split algebraic torus $T$, i.e. we have a short exact sequence
\begin{equation}\begin{tikzcd}\label{eq_TEB}
    0\arrow[r]& T \arrow[r] & E\arrow[r] &B\arrow[r] & 0 
\end{tikzcd}\end{equation}
of $K$-algebraic groups. 

The covering map $E^{an}\rightarrow A^{an}$, however, is not an algebraic morphism. Its kernel $M'=\ker(E^{an}\rightarrow A^{an})\simeq H_1(A^{an},\Z)$ is a finitely generated free abelian group and so we have a short exact sequence
\begin{center}\begin{tikzcd}
    0\arrow[r]& M' \arrow[r] & E^{an}\arrow[r] &A^{an}\arrow[r] & 0 
\end{tikzcd}\end{center}
of $K$-analytic groups. We may summarize this situation in terms of the \emph{Raynaud uniformization cross:}
\begin{equation*}
    \begin{tikzcd}
    & T^{an}\arrow[d]&\\
    M' \arrow[r]& E^{an}\arrow[r]\arrow[d]& A^{an} \ .\\
    & B^{an} & 
    \end{tikzcd}
\end{equation*}

Even more is true: Let $T_0$ be the affinoid torus in $T^{an}$. There is a unique compact analytic domain $A_0$ in $A^{an}$ that has the structure of formal $K$-analytic subgroup whose special fiber is an extension of an abelian variety $\overline{B}$ by a split algebraic torus $\overline{T}$. The abelian variety $\overline{B}$ is the special fiber of an abelian $R$-scheme model $\calB$ of $B$ and $\overline{T}$ is the special fiber of $T_0$. The short exact sequence
\begin{equation*}
    \begin{tikzcd}
        0\arrow[r] & \overline{T}\arrow[r]& \overline{A}\arrow[r]&\overline{B}\arrow[r]&0
    \end{tikzcd}
\end{equation*}
lifts to a short exact sequence
\begin{equation}\label{eq_T0A0B}
    \begin{tikzcd}
        0\arrow[r] & T_0\arrow[r]& A_0\arrow[r]&B^{an}\arrow[r]&0
    \end{tikzcd}
\end{equation}
of formal $K$-analytic groups and the short exact sequence \eqref{eq_TEB} is the pushout of \eqref{eq_T0A0B} along the inclusion $i\colon T_0\hookrightarrow T^{an}$.

\subsection{Duality and uniformization}\label{section_duality&uniformization} Now let $A'$ be the dual abelian variety of $A$. As explained in \cite[Section 6.3]{Luetkebohmert_curves&Jacobians}, its universal cover $E'$ is dual to $E$ and its Raynaud uniformization cross is given by 
\begin{equation*}
    \begin{tikzcd}
    & (T')^{an}\arrow[d]&\\
    M \arrow[r]& (E')^{an}\arrow[r]\arrow[d]& (A')^{an} \\
    & (B')^{an} & 
    \end{tikzcd}
\end{equation*}
where $T'$ and $B'$ are the duals of $T$ and $B$ respectively and $M$ is the kernel of $(E')^{an}\rightarrow (A')^{an}$. 

By \cite[Theorem A.2.8]{Luetkebohmert_curves&Jacobians} the finitely generated free abelian group $M$ is the character lattice of $T$ and, vice versa, $M'$ is the character lattice of $T'$. By \cite[Proposition 6.1.8]{Luetkebohmert_curves&Jacobians} there is a natural pairing
\begin{equation*}
    \langle.,.\rangle' \colon M'\times M\longrightarrow P_{B\times B'}
\end{equation*}
into the Poincar\'e bundle on $B\times B'$ such that the absolute value 
\begin{equation*}
    \langle.,.\rangle:=-\log\big\vert \langle .,.\rangle' \big\vert \colon M'\times M\longrightarrow \R
\end{equation*}
is non-degenerate. 

\subsection{Tropicalization and skeleton} Write $T=\Spec K[M]$. Set $N_\R=\Hom(M,\R)$ as well as $N_\R'=\Hom(M',\R)$. There is a natural continuous, proper, and surjective tropicalization map $\trop_T\colon T^{an}\rightarrow N_\R$ given by sending a point $x\in T^{an}$ to the homomorphism 
\begin{equation*}
    m\longmapsto -\log \vert \chi^m\vert_x \ ,
\end{equation*}
where $\chi^m$ is the character of $m$ in $K[M]$. We have $\trop^{-1}(0)=T_0$. Since 
\begin{equation}\label{eq_E=pushout}
    \begin{tikzcd}
        T_0\arrow[r,hookrightarrow]\arrow[d,hookrightarrow,"\subseteq"']& A_0\arrow[d]\\
        T^{an}\arrow[r] & E^{an}
    \end{tikzcd}
\end{equation}
is a pushout square, we may extend $\trop_{T}$ to a continous, proper, and surjective tropicalization map $\trop_E\colon E^{an}\rightarrow N_\R$ by declaring $\trop_{A_0}(x)=0\in N_\R$ for all $x\in A_0$. 

By \cite[Proposition 2.7.2 (c)]{Luetkebohmert_curves&Jacobians}, the restriction of $\trop_E$ to $M'$ is injective and its image defines a full rank lattice in $N_\R$, which we also denote by $M'$. A polarization of $A$ is given by an isogeny $A\rightarrow A'$. This induces a homomorphism $\lambda\colon M\rightarrow M'$ of finite index such that the bilinear form $\langle .,\lambda(.)\rangle$ is symmetric and non-degenerate.  So the integral real torus $\Sigma=\Sigma(A)=N_\R/M'$ is a tropical abelian variety. Moreover, a principal polarization is given by an isomorphism $A\xrightarrow{\sim}A'$, which, in turn, makes $\lambda\colon M'\xrightarrow{\sim} M$ into an isomorphism. 

The tropicalization map $\trop_E\colon E^{an}\rightarrow N_\R$ descends to a natural continuous, proper, and surjective tropicalization map $\trop_A\colon A^{an}\rightarrow \Sigma$. In \cite[Section 6.5]{Berkovich_book} Berkovich shows that there is a continuous section $J_A\colon \Sigma\rightarrow A^{an}$ of $\trop_A$ such that the composition 
\begin{equation*}
    \rho_A:=J_A\circ\trop_A\colon A^{an}\longrightarrow A^{an}
\end{equation*}
is a strong deformation retraction onto a closed subset of $A^{an}$ that is naturally homeomorphic to $\Sigma$, the \emph{non-Archimedean skeleton} of $A^{an}$. 

Denote the value group of $K$ by $H$. Given a closed subset $Y\subseteq A$, we define the \emph{tropicalization} of $Y$ to be 
\begin{equation*}
    \Trop(Y):= \trop_A(Y^{an}).
\end{equation*}
By Gubler's Bieri-Groves Theorem \cite[Theorem 6.9]{Gubler_trop&nonArch}, $\Trop(Y)$ has the structure of an $H$-rational polyhedral complex in $\Sigma$ of dimension at most $\dim(Y)$. If $Y$ is equidimensional and $A$ is totally degenerate, then we have $\dim(Y)=\dim\Trop_A(Y)$.

\subsection{Functoriality} \label{section_functoriality} Let $f\colon A_1\rightarrow A_2$ be a homomorphism of abelian varieties with split semiabelian reduction over $K$. Since both $E_1^{an}$ and $E_2^{an}$ are covering spaces, this homomorphism induces a homomorphism $\widetilde{f}\colon E_1\rightarrow E_2$ that makes the diagram
\begin{equation*}
    \begin{tikzcd}
        E_1^{an}\arrow[r,"\widetilde{f}^{an}"]\arrow[d] & E_2^{an}\arrow[d]\\
        A_1^{an}\arrow[r,"f^{an}"]& A_2^{an}
    \end{tikzcd}
\end{equation*}
commute, and which restricts to a homomorphism $f_{M'}\colon M_1'\rightarrow M_2'$ of the kernel lattices. Since 
\begin{equation*}
    \Hom(E_i,\G_m)=\Hom(T_i,\G_m) \ ,
\end{equation*} this induces a homomorphism $f_{M}\colon M_2\rightarrow M_1$ of character lattices. Thus the homomorphism $\widetilde{f}\colon E_1\rightarrow E_2$ restricts to a homomorphism $\widetilde{f}_T\colon T_1\rightarrow T_2$ and this induces a homomorphism $f_B\colon B_1\rightarrow B_2$ on the quotients $B_i=E_i/T_i$.

In the following well-known Proposition \ref{prop_functoriality} we summarize the functorial properties of the skeleton that will play a crucial role in the proof of Theorem \ref{thm_skeletonofPrym=Prymofskeleton} below. 

\begin{proposition}\label{prop_functoriality}
Let $f\colon A_1\rightarrow A_2$ be a homomorphisms of abelian varieties with split semiabelian reduction over $K$ and write $\Sigma_i=\Sigma(A_i)$ for $i=1,2$.
\begin{enumerate}[(i)]
    \item There is a unique integral homomorphism of tropical abelian varieties $\Sigma(f)\colon \Sigma_1\rightarrow\Sigma_2$ that makes the diagram
    \begin{equation*}
        \begin{tikzcd}
            A_1^{an} \arrow[rr,"f^{an}"]\arrow[d,"\trop_{A_1}"'] && A_2^{an}\arrow[d,"\trop_{A_2}"]\\
            \Sigma_1 \arrow[rr,"\Sigma(f)"] && \Sigma_2 
        \end{tikzcd}
    \end{equation*}
    commute. 
    \item The association $f\mapsto \Sigma(f)$ is functorial in $f$, i.e. we have $\Sigma(\id_{A})=\id_{\Sigma}$ and $\Sigma(f\circ g)=\Sigma(f)\circ\Sigma(g)$.
    
    \item If $f'\colon A'_2\rightarrow A_1'$ is the dual homomorphism to $f$, then the homomorphism $\Sigma(f')\colon\Sigma'_2\rightarrow\Sigma'_1$ is the dual homomorphism of $\Sigma(f)$, i.e. we have $\Sigma(f)'=\Sigma(f')$.
\end{enumerate}
\end{proposition}

\begin{proof}
Let $f\colon A_1\rightarrow A_2$ be a homomorphism of abelian varieties with split semiabelian reduction over $K$. The homomorphism $f_{M}\colon M_2\rightarrow M_1$ of abelian groups dualizes to an integral homomorphism $f_N\colon N^1_\R\rightarrow N_\R^2$ such that the diagrams
\begin{equation*}
    \begin{tikzcd}
        T_1^{an}\arrow[rr,"f_T^{an}"]\arrow[d,"\trop_{T_2}"']&& T_2^{an}\arrow[d,"\trop_{T_2}"]\\
        N_\R^1\arrow[rr,"f_N"]&& N_\R^2
    \end{tikzcd}\qquad \textrm{and}\qquad 
    \begin{tikzcd}
        E_1^{an}\arrow[rr,"f_T^{an}"]\arrow[d,"\trop_{E_1}"']&& E_2^{an}\arrow[d,"\trop_{E_2}"]\\
        N_\R^1\arrow[rr,"f_N"]&& N_\R^2
    \end{tikzcd}
\end{equation*}
commute. Here the commutativity of the diagram on the right follows fromt the commutativity of the diagram on the left, since \eqref{eq_E=pushout} is a pushout diagram. From \cite[Proposition 6.4.1 (a)]{Luetkebohmert_curves&Jacobians} it follows that for all  $m'\in M_1'$ and $m\in M_2$ we have
\begin{equation}\label{eq_duality}
    \big\langle m',f_{M}(m)\big\rangle_1 = \big\langle f_{M'}(m'), m\big\rangle_2 \ . 
\end{equation}
So the integral homomorphism $f_N\colon N_\R^1\rightarrow N_\R^2$ descends to a homomorphism $f_\Sigma\colon \Sigma_1\rightarrow \Sigma_2$ of tropical abelian varieties. The association $f\mapsto f_\Sigma$ is functorial and makes the induced diagram
\begin{equation*}
    \begin{tikzcd}
        A_1^{an}\arrow[rr,"f^{an}"]\arrow[d,"\trop_{A_1}"']&& A_2^{an}\arrow[d,"\trop_{A_2}"]\\
        \Sigma_1\arrow[rr,"f_\Sigma"] &&\Sigma_2
    \end{tikzcd}
\end{equation*}
commute. Similarly, the dual homomorphism $f'\colon A'_2\rightarrow A'_1$ gives rise to a homomorphism $f'_{\Sigma}\colon\Sigma'_2\rightarrow \Sigma'_1$ of the dual tropical abelian varieties 
coincides with the dual homomorphism of $f_\Sigma$ by equation \eqref{eq_duality}.
\end{proof}

Denote by $\ker(f)_0$ the connected component of the kernel of $f$ containing zero and by $\coker(f)$ the cokernel of $f$. The following Corollary \ref{cor_kernelskeleton} is a central ingredient in the proof of Theorem \ref{thm_skeletonofPrym=Prymofskeleton} in Section \ref{section_skeletonofPrym} below.

\begin{corollary}\label{cor_kernelskeleton}
Let $f\colon A_1\rightarrow A_2$ be a homomorphisms of polarized abelian varieties with split semiabelian reduction over $K$. 
\begin{enumerate}[(i)]
\item The skeleton of $\ker(f)_0$ is naturally isomorphic (as a polarized tropical abelian variety) to $\ker(\Sigma(f))_0$.
\item The skeleton of $\coker(f)_0$ is naturally isomorphic (as a polarized tropical abelian variety) to $\coker(\Sigma(f))$. 
\end{enumerate}
\end{corollary}

\begin{proof}
The abelian variety $\ker(f)_0$ has a Raynaud uniformization cross:
\begin{equation*}
    \begin{tikzcd}
    & \ker(f_T)_0^{an}\arrow[d]&\\
    \ker(f_{M'}) \arrow[r]& \big(E_{\ker(f)_0}\big)^{an}\arrow[r]\arrow[d]& \ker(f)_0^{an} \ .\\
    & \ker(f_B)_0^{an} & 
    \end{tikzcd}
\end{equation*}
Since there is a one-to-one correspondence between split algebraic tori and lattices, the tropicalization of $\ker(f_T)_0$ is naturally isomorphic to $\ker(f_N)$ and this isomorphism descends to an isomorphism $\ker(f_N)/\ker(f_{M'})\xrightarrow{\sim}\ker(f_\Sigma)_0$ of integral real tori. 
Let $\phi_i\colon A_i\rightarrow A_i'$ be a polarization of $A_i$ for $i=1,2$, such that $f'\circ\phi_2\circ f=\phi_1$. Then, by Proposition \ref{prop_functoriality}, we have 
\begin{equation*}
    \Sigma(\phi_1)=(f')_\Sigma \circ \Sigma(\phi_2)\circ f_\Sigma = (f_\Sigma)' \circ \Sigma(\phi_2) \circ f_\Sigma
\end{equation*}
and so the induced polarization of $\ker(f)_0$ tropicalizes to the induced polarization of $\ker(f_\Sigma)$. The argument for Part (ii) proceeds dually and is left to the avid reader. 
\end{proof}

\subsection{The skeleton of $\Jac(X)$}\label{section_skelJac=Jacskel} Let $X$ be a smooth projective curve over $K$. The Jacobian $\Jac(X)$ of $X$ is an abelian variety with split semi-abelian reduction whose uniformization is given by $\Jac(X)=E^{an}/M'$, where 
\begin{equation*}
    M'=H_1\big(\Jac(X)^{an}\big)=H_1(X^{an})=H_1(\Gamma_X) \ ,
\end{equation*}
since $\Gamma_X$ is the non-Archimedean skeleton of $X^{an}$. In \cite[Theorem 1.3]{BakerRabinoff_skelJac=Jacskel} Baker and Rabinoff prove that there is a canonical isomorphism
\begin{equation*}
\mu_{X}\colon\Jac(\Gamma_X) \xlongrightarrow{\simeq} \Sigma\big(\Jac(X)\big)
\end{equation*}
of principally polarized tropical abelian varieties that naturally commutes with the Abel-Jacobi maps, i.e. that makes the natural diagram
\begin{equation}\label{eq_skelJac=Jacskel}\begin{tikzcd}
X^{an}\arrow[rr,"\alpha_{q}"]\arrow[d,"\rho_X"']&&\Jac(X,\pi)^{an} \arrow[d,"\rho_{\Jac(X)}"] \\
  \Gamma_X\arrow[r,"\alpha_{\rho(q)}"]&\Jac(\Gamma_X)\arrow[r,"\sim","\mu_X"']&\Sigma\big(\Jac(X)\big) 
\end{tikzcd}\end{equation}
commute. The commutativity of \eqref{eq_skelJac=Jacskel} allows us to identify $\rho_{\Jac(X,\pi)}$ with Baker's specialization map $\trop_{\Pic_0(X)}\colon \Pic_0(X)^{an}\rightarrow \Pic_0(\Gamma_X)$ from \cite{Baker_specialization} given by pushing forward divisor to $\Gamma_X$, i.e. by $[D]\mapsto \big[\rho_{X,\ast}D\big]$ for a divisor class $[D]$ on $X_{K'}$ for a non-Archimedean extension $K'$ of $K$.

\subsection{Analytic and tropical norm maps}

Let $\pi\colon \widetilde{X}\rightarrow X$ be finite morphism. Then, as explained in \cite[Appendix B.1]{ACGHI}, there is a natural \emph{norm homomorphism} $\Nm_\pi\colon\Pic(\widetilde{X})\rightarrow \Pic(X)$ given by $\calO(D)\mapsto\calO(\pi_\ast D)$. We observe the following:

\begin{proposition}\label{prop_commcube}
Let $\widetilde{x}\in \widetilde{X}$ and set $x=\pi(\widetilde{q})$, $\widetilde{q}=\rho_{\widetilde{X}}(\widetilde{x})$, and $q=\rho_X(x)$. Then the diagram 
\begin{center}
    \begin{tikzcd}[row sep=2em, column sep = 2em]
    \widetilde{X}^{an} \arrow[rrr,"\alpha_{\widetilde{x}}^{an}"] \arrow[dr,"\pi^{an}"] \arrow[ddd,"\rho_{\widetilde{X}}"'] &&&
    \Jac(\widetilde{X})^{an} \arrow[ddd,"\rho_{\Jac(\widetilde{X})}"] \arrow[dr,"\Nm_\pi^{an}"] \\
    & X^{an} \arrow[rrr,"\alpha_x^{an}\ \ \ \ \ \ \ \ \ "]\arrow[ddd,"\rho_X"'] &&&
    \Jac(X)^{an} \arrow[ddd,"\rho_{\Jac(X)}"] \\ \\
    \Gamma_{\widetilde{X}}\arrow[rrr,"\ \ \ \ \ \ \ \ \ \alpha_{\widetilde{q}}"] \arrow[dr, "\pi^{trop}"'] &&& \Jac(\Gamma_{\widetilde{X}}) \arrow[dr,"\Nm_{\pi^{trop}}"] \\
    & \Gamma_X \arrow[rrr,"\alpha_{q}\ "] &&& \Jac(\Gamma_X)
    \end{tikzcd}
    \end{center}
commutes.
\end{proposition}

\begin{proof}  
The top square commutes, because $(.)^{an}$ is a functor, the bottom square commutes by Proposition  \ref{prop_NormAbelJacobi}, the left square commutes by \cite[Theorem A]{ABBRI}, the front and back squares commute by \cite[Theorem 1.3]{BakerRabinoff_skelJac=Jacskel}, i.e. by the commutativity of \eqref{eq_skelJac=Jacskel}. This implies that also the square on the right commutes. 
\end{proof}

By Proposition \ref{prop_functoriality} (i), there is an induced homomorphism 
\begin{equation*}
    \Sigma(\Nm_{\pi})\colon \Sigma\big(\Jac(\widetilde{X})\big)\longrightarrow \Sigma\big(\Jac(X)\big) \ .
\end{equation*} The following Corollary \ref{cor_tropnorm} shows that this construction agrees with the tropical norm map defined in Section \ref{section_tropicalnormmapI}. 

\begin{corollary}
\label{cor_tropnorm}
For a finite morphism $\pi\colon \widetilde{X}\rightarrow X$, 
the induced map $\Sigma(\Nm_\pi)$ is equal to $\Nm_{\pi^{trop}}$.
\end{corollary}

\begin{proof}
By Proposition \ref{prop_functoriality} (i), the induced homomorphism $\Sigma(\Nm_{\pi})\colon \Sigma\big(\Jac(\widetilde{X})\big)\rightarrow \Sigma\big(\Jac(X)\big)$ is unique and so we find, using Proposition \ref{prop_commcube}, that  $\Sigma(\Nm_\pi)=\Nm_{\pi^{trop}}$.
\end{proof}


\subsection{Skeletons of (generalized) Prym varieties} 
Let $\pi\colon \widetilde{X}\rightarrow X$ be a finite morphism. We define the \emph{generalized Prym variety} $\Pr(X,\pi)$ associated to this datum to be the component of $\ker(\Nm_\pi)$ containing zero. Given a finite harmonic morphism $\pi\colon \widetilde{\Gamma}\rightarrow \Gamma$ of tropical curves, we define in complete analogy the \emph{higher tropical Prym variety} $\Pr(\Gamma, \pi)$ to be the component of the kernel of $\ker(Nm_\pi)$ containing zero. Both the algebraic and the tropical higher Prym variety naturally carry the induced polarization from $\Jac(\widetilde{X})$ and $\Jac(\widetilde{\Gamma})$ respectively. 

The following Theorem \ref{thm_skeletonofhigherPrym=higherPrymofskeleton} partially generalizes Theorem \ref{thm_skeletonofPrym=Prymofskeleton} to higher Prym varieties.

\begin{theorem}\label{thm_skeletonofhigherPrym=higherPrymofskeleton}
Let $\pi\colon \widetilde{X}\rightarrow X$ be a finite homomorphism. There is a natural isomorphism
\begin{equation*}
\mu_{X,\pi}\colon \Pr(\Gamma_X, \pi^{trop}) \xlongrightarrow{\simeq} \Sigma\big(K_0(X,\pi)\big)
\end{equation*}
of polarized tropical abelian varieties that makes the diagram
\begin{center}\begin{tikzcd}
X^{an}\arrow[rrr,"\rho_X"]\arrow[d,"\alpha_{X,\pi}^{an}"']&&& \Gamma_X\arrow[d,"\alpha_{\Gamma,\pi^{trop}}"]\\
\Pr(X,\pi)^{an}\arrow[rr, "\rho_{\Pr(X,\pi)}"] && \Sigma\big(\Pr(X,\pi)\big) & \Pr(\Gamma_X, \pi^{trop})\arrow[l,"\sim"',"\mu_{X,\pi}"]
\end{tikzcd}\end{center}
commute.  
\end{theorem}

\begin{proof}
By Corollary \ref{cor_tropnorm}, the induced map $\Sigma(\pi)\colon \Sigma(\Jac(\widetilde{X}))\rightarrow \Jac(\Sigma(X))$ may be identified with the tropical norm map $\Nm_{\pi^{trop}}\colon \Jac(\widetilde{\Gamma})\rightarrow \Jac(\Gamma)$. Consequently, the tropical Prym variety $\Pr(\Gamma,\pi)$ is equal to the zero component of the kernel of $\Sigma(f)$, which, by Corollary \ref{cor_kernelskeleton}, is equal to the skeleton of $\Pr(X,\pi)^{an}$. Finally, Proposition \ref{cor_tropnorm} together with the natural compatibility of tropicalization with both the Abel-Jacobi map, as proved in \cite[Theorem 1.3]{BakerRabinoff_skelJac=Jacskel} (also see Section \ref{section_skelJac=Jacskel} above), and with dual homomorphisms, as proved in Proposition \ref{prop_functoriality} (iii), implies that the retraction $\rho_{X,\pi}$ to the skeleton commutes with the Abel-Prym map. 
\end{proof}

We conclude this section with the proof of Theorem \ref{thm_skeletonofPrym=Prymofskeleton}.

\begin{proof}[Proof of Theorem \ref{thm_skeletonofPrym=Prymofskeleton}]
Let $\pi\colon 
\widetilde{X}\rightarrow X$ be an unramified double cover. By \cite{Mumford_Prym}, there is a principal polarization $\Xi$ on $\Pr(X,\pi)$ such that $i^\ast \widetilde{\Theta}=2\cdot \Xi$. Denote by $i^\ast\phi_{\widetilde{\Theta}}$ and $\phi_\Xi$ the corresponding maps $\Pr(X,\pi)\rightarrow\Pr(X,\pi)'$ to their duals. Then $i^\ast \widetilde{\Theta}=2\cdot \Xi$ can be rephrased as saying that $i^\ast\phi_{\widetilde{\Theta}}=\phi_\Xi^2$. 

Denote the toric parts of the universal covers of both $\Jac(\widetilde{X})$ and $\Jac(X)$ by $\widetilde{T}$ and $T$ respectively. Then we have $\widetilde{T}=H_1(\widetilde{\Gamma})\otimes\GG_m$ and $T=H_1(\Gamma)\otimes\GG_m$. By Corollary \ref{cor_tropnorm} and Corollary \ref{cor_norm=dualtopullback}, the morphism induced by the norm map $\Nm_\pi\colon\Jac(\widetilde{X})\rightarrow\Jac(X)$ is the one induced by the pushforward map  $\pi_\ast\colon H_1(\widetilde{\Gamma})\rightarrow H_1(\Gamma)$. So, choosing spanning trees as in Lemma \ref{lemma_spanningtrees}, we see that $\phi_\Xi$ can only induce the principal polarization $\lambda_\Xi$ from Theorem \ref{thm_Prymvariety}. 
\end{proof}
  



\section{Tropical Prym--Brill--Noether theory}
In this section, we discuss the theory of special divisors on a torsor of the Prym variety.

 \begin{definition}
Let $\Gamma$ be a tropical curve of genus $g$ and $\pi\colon\widetilde{\Gamma}\rightarrow\Gamma$ an unramified double cover. The \emph{Prym--Brill--Noether} locus associated with $\pi$ is the collection of divisor classes that map down to the canonical divisor of $\Gamma$.  Explicitly,
\begin{equation*}
 V = V(\Gamma,\pi)=\big\{
 [D]\in\Pic_{2g-2}(\widetilde{\Gamma})\big\vert   \Nm_\pi[D]=[K_{\Gamma}] 
 \big\}.
 \end{equation*}
\end{definition}

We refer to divisors whose class is in $V$ as \emph{Prym} divisors. 
The algebraic version of the Prym--Brill--Noether locus consists of two connected components. 
In nice cases, the analogous tropical statement is  true as well.
\begin{proposition}\label{prop:PBN_components}
If $\pi$ is a topological double cover, then $V$ is a disjoint union of two connected components. Otherwise, $V$ consists of a single component. 
\end{proposition}
\begin{proof}
By  Theorem \ref{thm_Prymvariety} above, the Prym variety $\Pr(\Gamma,\pi)$ consists of two connected components when $\pi$ is a topological cover, and of a single component otherwise. The same is true for $V$, since it is a translation of $\Pr(\Gamma,\pi)$ by any pre-image of $K_{\Gamma}$.
\end{proof}

In  the algebraic case, the two components of the locus correspond to parities of the rank of the divisor classes. The following example shows that the analogous statement may not hold in the tropical setting, even for topological covers.

\begin{example}\label{ex:rank_jump}
Consider the double cover in Figure \ref{fig:rank_jump}. Let $D_n$ be the divisor in red, such that the distance of the upper left (resp. lower right) chip from the upper (resp. lower) vertex is $\frac{1}{n}$. Then each $D_n$ is in the Prym--Brill--Noether locus and their rank is $0$, but the rank of $D=\lim_{n\to\infty} D_n$ is $1$. 
\end{example}

\begin{figure}[h!]
\begin{tikzpicture}

\draw (-2.6,1.2) node  {${\widetilde\Gamma}$};
\draw (-1,1.2) circle (.4);
\draw (-1,0.8)--(1,0.8);
\draw (-1,1.6)--(1,1.6);
\draw (1,1.2) circle (.4);

\draw[fill=red] (-1.35,1.35) circle (.05);
\draw[fill=red] (-.65,1.05) circle (.05);

\draw (-2.6,0) node  {${\Gamma}$};
\draw (-.8,0) circle (.2);
\draw (-0.6,0)--(0.6,0);
\draw (.8,0) circle (.2);

\end{tikzpicture}
\caption{Divisors of rank $0$ converging to a divisor of rank $1$.}
\label{fig:rank_jump}
\end{figure}
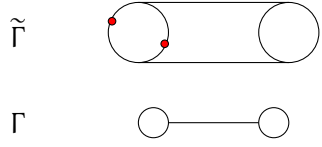

However, we will see in the remainder of this section that the Prym--Brill--Noether locus is well-behaved in the special case  a \emph{folded chain of loops}.

\subsection{Young tableaux and divisors on chains of loops} 
When our metric graph is a so-called \emph{chain of loops}, there is an elegant correspondence between divisor classes and rectangular Young tableau, which we now describe. Throughout, we use the  French style to discuss partitions and Young tableau, as in \cite{Pflueger_kgonalcurves}, rather than the English style that appeared in  \cite[Section 2]{JensenRanganathan_BrillNoetherwithfixedgonality}.
For instance, the $(1,1)$-cell of a Young tableau is in the bottom-left corner, and the cell $(2,1)$ is located one step to the right.  Given an $m\times m$ partition, we refer to the cells with coordinates $(a,b)$ with $a=b$  as the \emph{diagonal}, and those where $a+b=m+1$ as the \emph{anti-diagonal}. For a tableau $t$, we denote $t(a,b)$ the symbol in the cell $(a,b)$. 

We set up some notation. 
Let $\Lambda$ be a chain of $g(\Lambda)$ loops as in Figure \ref{Fig:Chain_of_loops}. 

\begin{figure}[h!]
\begin{tikzpicture}

\begin{scope}[scale=.7]

\draw (-3.5,0.7) node {$\Lambda$};

\draw (-2.7,-.5) node {\tiny $w_0$};
\draw [ball color=black, color=black] (-2.5,-0.3) circle (0.55mm);
\draw (-2.5,-.3)--(-1.93,-0.25);
\draw [ball color=black, color=black] (-1.93,-0.25) circle (0.55mm);
\draw (-1.95,-0.5) node {\tiny $v_1$};
\draw (-1.5,0) circle (0.5);
\draw (-1.5,0) node {\tiny $\lambda_1$};

\draw (-1,0)--(0,0.5);
\draw [ball color=black, color=black] (-1,0) circle (0.55mm);
\draw (-0.85,0.3) node {\tiny $w_1$};
\draw (0.7,0.5) circle (0.7);
\draw (1.4,0.5)--(2,0.3);
\draw [ball color=black, color=black] (1.4,0.5) circle (0.55mm);
\draw [ball color=black, color=black] (0,0.5) circle (0.55mm);
\draw (-0.2,0.75) node {\tiny $v_2$};

\draw (2.6,0.3) circle (0.6);
\draw (2.6,0.3) node {\tiny $\lambda_k$};
\draw [<->] (3.3,0.4) arc[radius = 0.715, start angle=10, end angle=170];
\draw [<->] (3.3,0.2) arc[radius = 0.715, start angle=-9, end angle=-173];
\draw (2.5,1.25) node {\footnotesize$\ell_k$};
\draw (2.75,-0.7) node {\footnotesize$m_k$};

\draw (3.2,0.3)--(3.87,0.6);
\draw [ball color=black, color=black] (2,0.3) circle (0.55mm);
\draw [ball color=black, color=black] (3.2,0.3) circle (0.55mm);
\draw [ball color=black, color=black] (3.87,0.6) circle (0.55mm);
\draw (4.5,0.3) circle (0.7);
\draw (5.16,0.5)--(5.9,0);

\draw (6.1,0) circle (0.2);
\draw [ball color=black, color=black] (5.9,0) circle (0.55mm);
\draw [ball color=black, color=black] (6.3,0) circle (0.55mm);

\draw (6.3,0)--(7.2,0);
\draw [ball color=black, color=black] (7.2,0) circle (0.55mm);
\draw (7.7,0) circle (0.5);
\draw (7.7,0) node {\tiny $\lambda_g$};

\draw [ball color=black, color=black] (8.2,0) circle (0.55mm);
\draw (8.2,0)--(8.5,0.2);
\draw [ball color=black, color=black] (8.5,0.2) circle (0.55mm);
\draw (8.7,0.5) node  {\tiny $v_{g+1}$};

\end{scope}

\end{tikzpicture}
\caption{A chain of $g$ loops}
\label{Fig:Chain_of_loops}
\end{figure}
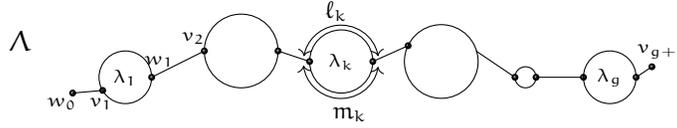

Denote $\ell_i$ and $m_i$ the lengths of the upper and lower arcs  of each loop $\lambda_i$ respectively.   
The \emph{torsion} of the loop is the smallest positive integer $s$ such that $s\cdot m_i$ is an integer multiple of $\ell_i+m_i$. If no such integer exists, then the torsion is $0$. 
From now on, assume for simplicity that $m_i=1$ for each $i$ (this will have no bearing on the results in this paper).
Each loop $\lambda_i$ has a vertex $v_i$ that is closest to $\lambda_{i-1}$ (referred to as the \emph{tail} vertex), and a vertex $w_i$ that is closest to $\lambda_{i+1}$ (referred to as the \emph{head} vertex). Moreover, there is an edge from $v_1$ the a vertex $w_0$ and an edge between $w_g$ to a vertex $v_{g+1}$.

Fix integers $r,d$, and $h=g(\Lambda)-d+r$, and let $\lambda$ be a partition with $r+1$ rows and $h$ columns. A rectangular tableau $t$ on $\lambda$ is called a \emph{displacement tableau} if it is filled with  integers between $1$ and $g(\Lambda)$,   subject to the following condition: if a number $t(a,b)=t(c,d)=n$, and the torsion of the $n$-th loop is $s$, then the  lattice distance between $(a,b)$ and $(c,d)$ cells equals $0\mod s$. 

Such a tableau gives rise to a set $T(t)$ of divisor classes of degree $d$ and rank at least $r$ on $\Lambda$ as follows. 
The location $(a,b)$ of the number $n$ in the tableau specifies where to place a chip on the $n$-th loop. Whenever a number $n\in\big\{1,2,\ldots, g(\Lambda)\big\}$ does not appear in the tableau, the chip may be placed arbitrarily on the loop. 
Otherwise, let $c=b-a$. Then the $n$-th loop will have a chip at distance $c$ counter-clockwise from its head vertex $w_n$. For instance, if $n$ is on the diagonal, then the chip will be on the head vertex, if $n$ is one cell left of the diagonal then the chip will be on the tail vertex $v_{n}$, and if $n$ is one cell to the right, then the chip will be at distance $1$ \emph{counter-clockwise} from the head vertex. Note that this is well defined even if $n$ repeats in the tableau, thanks to the torsion condition.   
Finally, place $d-g(\Lambda)$ chips on $v_{g+1}$ to obtain a divisor of degree $d$ (this number may be either positive or negative). 

According to \cite[Theorem 1.4]{Pflueger_special}, the classes of divisors thus constructed are precisely the divisors of rank $r$ on $\lambda$. Namely
\begin{equation}\label{eq:BN_locus}
W^r_d(\Lambda) = \bigcup_{t\vdash\lambda} T(t),
\end{equation}
where $\lambda$ is a rectangle partition of height $r+1$ and width $d-g+r$. 
Explicit examples of this construction will be given when we specialize to the case of double covers.

\subsection{Folded chains of loops}\label{section_generic2:1chainofloops}
Let $\Gamma$ be the graph obtained from a chain of $g$ loops after removing the vertex $v_{g+1}$ and the edge leading to it.  
We construct a double cover as in Figure \ref{Fig:Folded_chain_of_loops}. Explicitly, denote $\gamma_1,\ldots,\gamma_g$ the loops of $\Gamma$, and let  $\widetilde\Gamma$ be a chain of $2g-1$ loops, denoted $\widetilde{\gamma}_1,\ldots,\widetilde{\gamma}_{2g-1}$. For $i=1,2,\ldots,g-1$, the edge lengths of $\widetilde{\gamma}_i$ and $\widetilde{\gamma}_{2g-i}$ are chosen to equal the edge length of $\gamma_i$.  As for the loop $\widetilde{\gamma}_g$, each of its edges will have the same length as the loop $\gamma_g$. There is a natural double cover $\pi:\widetilde\Gamma\to\Gamma$ by letting $\widetilde{\gamma}_i$ and $\widetilde{\gamma}_{2g-i}$ cover $\gamma_i$  for $i=1\ldots,g-1$, and letting $\widetilde{\gamma}_g$ cover $\gamma_g$ twice. We refer to $\pi$ as a \emph{folded chain of loops}.

In what follows, we are interested in divisors of degree $2g-2$ and rank at least $r$ on $\widetilde\Gamma$. In this case,   $h=g(\widetilde\Gamma)-d+r = 2g-1-(2g-2)+r= r+1$, and such divisors correspond to $(r+1)\times(r+1)$ tableaux. The loop $\widetilde{\gamma}_g$ has torsion $2$, so the number $g$ may repeat in the tableau, as long as the lattice distance between any two occurrences is even.  The torsion of each loop $\widetilde{\gamma}_i$ equals the torsion of $\widetilde{\gamma}_{2g-i}$.

\begin{figure}[h!]
\begin{tikzpicture}[scale=.7,every node/.style={scale=.7}] 

\begin{scope}[shift={(0,2.6)}]

\draw (-3.5,-1.3) node {$\widetilde\Gamma$};

\draw (-2.7,-.5) node {\footnotesize $\widetilde{v}_{2g}$};
\draw [ball color=black, color=black] (-2.5,-0.3) circle (0.55mm);
\draw (-2.5,-.3)--(-1.93,-0.25);
\draw [ball color=black, color=black] (-1.93,-0.25) circle (0.55mm);
\draw (-1.5,0) circle (0.5);
\draw (-1.5,0) node {\tiny $\widetilde{\gamma}_{2g-1}$};

\draw (-1,0)--(0,0.5);
\draw [ball color=black, color=black] (-1,0) circle (0.55mm);
\draw (0.7,0.5) circle (0.7);
\draw (-0.4,0.75) node {\footnotesize $\widetilde{w}_{2g-2}$};

\draw (1.4,0.5)--(2,0.3);
\draw [ball color=black, color=black] (1.4,0.5) circle (0.55mm);
\draw [ball color=black, color=black] (0,0.5) circle (0.55mm);
\draw (2.6,0.3) circle (0.6);
\draw (3.2,0.3)--(3.87,0.6);
\draw [ball color=black, color=black] (2,0.3) circle (0.55mm);
\draw [ball color=black, color=black] (3.2,0.3) circle (0.55mm);
\draw [ball color=black, color=black] (3.87,0.6) circle (0.55mm);
\draw (4.5,0.3) circle (0.7);
\draw (5.16,0.5)--(5.9,0);
\draw (6.4,0) circle (0.5);
\draw [ball color=black, color=black] (5.16,0.5) circle (0.55mm);
\draw (5.48,0.74) node {\footnotesize $\widetilde{v}_{g+2}$};
\draw [ball color=black, color=black] (5.9,0) circle (0.55mm);
\draw [ball color=black, color=black] (6.9,0) circle (0.55mm);
\draw (7.27,.17) node {\footnotesize $\widetilde{v}_{g+1}$};

\draw [<->] (3.3,0.4) arc[radius = 0.715, start angle=10, end angle=170];
\draw [<->] (3.3,0.2) arc[radius = 0.715, start angle=-9, end angle=-173];

\draw (2.5,1.25) node {\footnotesize$\ell_k$};
\draw (2.75,-0.7) node {\footnotesize$m_k$};

\draw (6.9,0) arc[radius = .9, start angle=90, end angle=0];
\draw (7.8,-1.3) circle (0.4);
\draw [ball color=black, color=black] (7.8,-0.9) circle (0.55mm);
\draw [ball color=black, color=black] (7.8,-1.7) circle (0.55mm);

\draw (6.9,-2.6) arc[radius = .9, start angle=-90, end angle=0];
\draw (8.1,-0.8) node  {\footnotesize $\widetilde{w}_g$};
\draw (8.1,-1.85) node  {\footnotesize $\widetilde{v}_g$};
\draw (7.8,-1.3) node  {\tiny $\widetilde{\gamma}_g$};

\end{scope}

\begin{scope}
\draw (-2.7,-.5) node {\footnotesize $\widetilde{w}_0$};
\draw [ball color=black, color=black] (-2.5,-0.3) circle (0.55mm);
\draw (-2.5,-.3)--(-1.93,-0.25);
\draw [ball color=black, color=black] (-1.93,-0.25) circle (0.55mm);
\draw (-1.95,-0.5) node {\footnotesize $\widetilde{v}_1$};
\draw (-1.5,0) circle (0.5);
\draw (-1.5,0) node {\tiny $\widetilde{\gamma}_{1}$};

\draw (-1,0)--(0,0.5);
\draw [ball color=black, color=black] (-1,0) circle (0.55mm);
\draw (-0.85,0.3) node {\footnotesize $\widetilde{w}_1$};
\draw (0.7,0.5) circle (0.7);
\draw (1.4,0.5)--(2,0.3);
\draw [ball color=black, color=black] (1.4,0.5) circle (0.55mm);
\draw [ball color=black, color=black] (0,0.5) circle (0.55mm);
\draw (-0.2,0.75) node {\footnotesize $\widetilde{v}_2$};
\draw (2.6,0.3) circle (0.6);
\draw (3.2,0.3)--(3.87,0.6);
\draw [ball color=black, color=black] (2,0.3) circle (0.55mm);
\draw [ball color=black, color=black] (3.2,0.3) circle (0.55mm);
\draw [ball color=black, color=black] (3.87,0.6) circle (0.55mm);
\draw (4.5,0.3) circle (0.7);
\draw (5.16,0.5)--(5.9,0);
\draw (6.4,0) circle (0.5);
\draw [ball color=black, color=black] (5.16,0.5) circle (0.55mm);
\draw [ball color=black, color=black] (5.9,0) circle (0.55mm);
\draw [ball color=black, color=black] (6.9,0) circle (0.55mm);
\draw (5.6,-.2) node {\footnotesize $\widetilde{v}_{g-1}$};
\draw (7.3,-.2) node {\footnotesize $\widetilde{w}_{g-1}$};

\draw [<->] (3.3,0.4) arc[radius = 0.715, start angle=10, end angle=170];
\draw [<->] (3.3,0.2) arc[radius = 0.715, start angle=-9, end angle=-173];

\draw (2.5,1.25) node {\footnotesize$\ell_k$};
\draw (2.75,-0.7) node {\footnotesize$m_k$};
\end{scope}

\begin{scope}[shift={(0,-3)}]

\draw (-3.5,0) node {$\Gamma$};

\draw (-2.7,-.5) node {\footnotesize $w_0$};
\draw [ball color=black, color=black] (-2.5,-0.3) circle (0.55mm);
\draw (-2.5,-.3)--(-1.93,-0.25);
\draw [ball color=black, color=black] (-1.93,-0.25) circle (0.55mm);
\draw (-1.95,-0.5) node {\footnotesize $v_1$};
\draw (-1.5,0) circle (0.5);
\draw (-1.5,0) node {\footnotesize $\gamma_1$};

\draw (-1,0)--(0,0.5);
\draw [ball color=black, color=black] (-1,0) circle (0.55mm);
\draw (-0.85,0.3) node {\footnotesize $w_1$};
\draw (0.7,0.5) circle (0.7);
\draw (1.4,0.5)--(2,0.3);
\draw [ball color=black, color=black] (1.4,0.5) circle (0.55mm);
\draw [ball color=black, color=black] (0,0.5) circle (0.55mm);
\draw (-0.2,0.75) node {\footnotesize $v_2$};
\draw (2.6,0.3) circle (0.6);
\draw (3.2,0.3)--(3.87,0.6);
\draw [ball color=black, color=black] (2,0.3) circle (0.55mm);
\draw [ball color=black, color=black] (3.2,0.3) circle (0.55mm);
\draw [ball color=black, color=black] (3.87,0.6) circle (0.55mm);
\draw (4.5,0.3) circle (0.7);
\draw (5.16,0.5)--(5.9,0);
\draw (6.4,0) circle (0.5);
\draw [ball color=black, color=black] (5.16,0.5) circle (0.55mm);
\draw [ball color=black, color=black] (5.9,0) circle (0.55mm);
\draw [ball color=black, color=black] (6.9,0) circle (0.55mm);
\draw (7.27,-.2) node {\footnotesize $w_{g-1}$};
\draw (6.4,0) node {\tiny $\gamma_{g-1}$};

\draw (6.9,0)--(7.8,0);
\draw [ball color=black, color=black] (7.8,0) circle (0.55mm);
\draw (8,0) circle (0.2);
\draw (7.65,0.2) node  {\footnotesize $v_g$};

\draw [<->] (3.3,0.4) arc[radius = 0.715, start angle=10, end angle=170];
\draw [<->] (3.3,0.2) arc[radius = 0.715, start angle=-9, end angle=-173];

\draw (2.5,1.25) node {\footnotesize$\ell_k$};
\draw (2.75,-0.7) node {\footnotesize$m_k$};
\end{scope}

\draw [->, thick] (9,0.5)--(9,-2);
\draw [right] (9,-.75) node {\footnotesize $\pi$};

\end{tikzpicture}
\caption{A folded chain of $2g-1$ loops double covering a chain of $g$ loops.}
\label{Fig:Folded_chain_of_loops}
\end{figure}
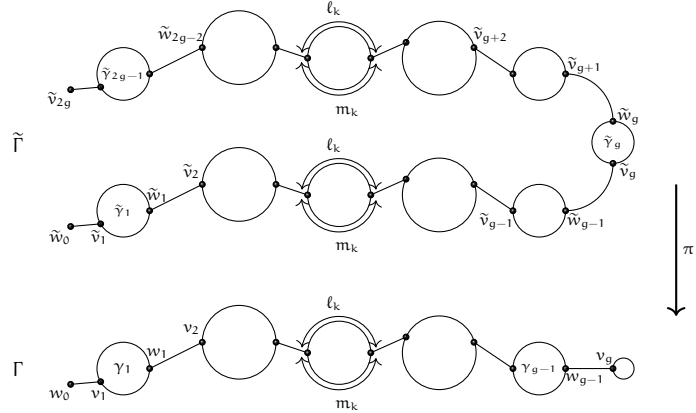

Note that $d-g(\widetilde\Gamma) = -1$,  so a divisor constructed as above  is effective everywhere, apart from a pole at $\widetilde{w}_{2g-1}$. Its image is a divisor on $\Gamma$ that has pole at $v_0$ and two chips on each loop, except for $\gamma_g$ where it has a single chip. 
Using the following lemma, we can describe the subset of these divisors that map down by $\pi_*$ to  $K_\Gamma$.

\begin{lemma}\label{lem:equivalenToCanonical}
Let $D$ be a divisor of degree $2g-2$ on $\Gamma$ such that $D+w_0$  is effective, $\deg{D|_{\gamma_i}} = 2$ for $1\leq i <g$, and $\deg{D|_{\gamma_i}} = 1$ for $i=g$. Then $D$ is equivalent to $K_{\Gamma}$  if and only if $D|_{\gamma_i}= p_i+q_i$, where the counter clockwise distance of $p_i$ from $v_i$ equals the clockwise distance of $q_i$ from $w_i$ for $1\leq i\leq g-1$, and $D_{\gamma_g} = v_{g-1}$. 
\end{lemma}

\begin{proof}
It's straightforward to see that the condition is sufficient. To see that it is also necessary, write $K_\Gamma =-w_0 + L + R$, where $L = v_1+\ldots+v_{g-1}$ consists of a chip at the left vertex of every loop,  and $R = w_1 + \ldots + w_{g-1}$ consists of a chip at the right vertex of every loop. Denote $D|_{\gamma_i} = p_i+q_i$ for $i=1\ldots,g-1$ and $D|_{\gamma_g} = p_g$. Then  $D-(R-w_0)$ is equivalent to an effective divisor $D'$, by moving $p_i$ to $v_i$, and moving $q_i$ an equal distance in the opposite direction. Note that $D'$ has a single chip on each loop. Since $D\simeq K_{\Gamma}$ it follows that $D'\simeq L$. But divisors with a single chip on each loop are uniquely determined by the position of the chip, so $D'=L$ on the nose. 
It follows that $p_g$ is on the vertex of $\gamma_g$, and the clockwise distance of $q_i$ from $w_i$ equals the counter-clockwise of $p_i$ from $v_i$ for every other loop. 
\end{proof}



Next, we wish to describe the structure of the Prym--Brill--Noether locus for a chain of loops. 
As we noted before, the locus consists of two connected components. Each divisor class in $V(\widetilde\Gamma)$ has a unique representative with a single chip on each loop, and $-1$ chips on $\widetilde{w}_{2g-1}$.
The counter-clockwise distance of the chip on the $n$-loop from $w_n$ is denoted by  $\langle\xi\rangle_n$. 
By Lemma \ref{lem:equivalenToCanonical} the divisor is Prym if and only if this representative is symmetric in the sense that $\langle\xi\rangle_i = \langle\xi\rangle_{2g-i}$ for each $i\neq g$, and $\langle\xi\rangle_g\in \{0,1\}$ (that is, the chip on $\gamma_g$ will either be on $\widetilde{v}_g$ or $\widetilde{w}_g$). 

\begin{definition}
A Prym divisor is said to be \emph{odd} if $\langle\xi\rangle_g=0$ and \emph{even} if $\langle\xi\rangle_g=1$.
\end{definition}

\begin{proposition}
The Prym--Brill--Noether locus $V(\Gamma,\pi)$ has two connected components. One consists of even divisors, and the other consists of odd divisors.
\end{proposition}
\begin{proof}
Since $\pi$ is a topological double cover, Proposition \ref{prop:PBN_components} implies that $V$ has two connected components, and it remains to show that they correspond to parity. 
The map $f\colon\Pic_{2g-1}(\widetilde\Gamma)\to (\RR/\ZZ)^{2g-1}$ which sends a divisor class $[D]$ to its coordinates $\langle\xi\rangle_i$ (for $ i=1,\ldots,{2g-1}$) defines a continuous bijection \cite[Lemma 3.3]{Pflueger_special}. Therefore, the map obtained by restricting $f$ to $V$ and composing with the projection onto the $g^{th}$ coordinate is continuous as well. It follows that the odd and even divisors must be in separate connected components of $V$. 
 \end{proof}
 
\subsection{Parity and rank} Denote by  $V^{-1} = V^{\text{even}}$ the set of even divisors, and $V^0 = V^{\text{odd}}$ the set of odd divisors. In this section we are going to show that 
odd Prym divisors always have even rank, whereas even Prym divisors have odd rank. The reason for this confusing terminology is that classically, the parity refers to $h^0=r+1$.

\begin{definition}
A displacement tableau $t$ on a rectangular partition  is said to be \emph{Prym} if $a-b\equiv c-d\mod s$ whenever $n=t(a,b) = 2g-t(c,d)$, and the torsion of the $n$-th loop is $s$. 
\end{definition}

Let $\epsilon\in\{-1,0\}$.
The  \emph{Prym--Brill--Noether cell} corresponding to $t$, denoted $P_{\epsilon}(t)$, is the subset of $T(t)$ consisting of  Prym divisors of parity $\epsilon$.  That is,
$P_\epsilon(t)$ is the subset of $T(t)$ of divisors whose coordinates satisfy $\langle\xi\rangle_i = \langle\xi\rangle_{2g-i}$ for each $i\neq g$ and $\langle\xi\rangle_g = \epsilon$. Moreover, denote $P(t) = P_{-1}(t)\cup P_{0}(t)$.
Note that if $t(a,b)=g$ and $b-a\not\equiv \epsilon\mod 2$, then $P_\epsilon(t)=\emptyset$, but if $g\notin t$ then both $P_{-1}(t)$ and $P_0(t)$ may be non-empty.

As the next example shows, even if $t$ is of length $r+1$, it is possible for $P(t)$ to only consist of divisors of rank strictly greater than $r$. 
\begin{example}\label{ex:higher_rank}
Let $g=4$, and consider the Prym tableau 
\begin{center}
t = \begin{tabular}{ |c|c| } 
 \hline
 3 & 4 \\ \hline
 1 & 2\\ 
 \hline
\end{tabular}.
\end{center}
Since the tableau has length $2$, one might be inclined to think that $P_0(t)$ consists of divisors of rank $1$. However, we claim that all the  divisors in $P(t)$ have rank at least $2$. To see that, note that the Prym tableau 
\begin{center}
s = \begin{tabular}{ |c|c|c| } 
 \hline
 4 & 5 & 7 \\ \hline
 3 & 4 & 6 \\ \hline
 1 & 2 & 4 \\ 
 \hline
\end{tabular}
\end{center}
imposes the same conditions on Prym divisors as $t$, so $P(s)=P(t)$. But every divisor in $P(s)$ has rank at least $2$ because $s$ is defined on a partition of length $3$.
\end{example}

For the rest of this section, we describe a method for constructing the tableau $s$ from $t$ as in Example \ref{ex:higher_rank}. 
We begin by recalling a useful construction from \cite[Section 2]{Pflueger_special}.

\begin{definition}\label{def:upwards_displacement}
Let $\lambda$ be a partition, and let $S$ be a set of integers. The \emph{upwards displacement} of $\lambda$ by $S$ is the partition
\[
\text{disp}^+(\lambda,S) = \lambda\cup L,
\]
where $L$ is the set of boxes $(x,y)\notin \lambda$, such that $x-y\in S$ and $(x,y-1),(x-1,y)\in \overline{\lambda}$ (where $\overline{\lambda} =  \lambda\cup (0\times\ZZ)\cup(\ZZ\times 0$)). 
\end{definition}

\begin{example} Suppose that 
$\lambda$ is the partition on the left of Figure \ref{fig:partition_and_displacement}, and $S=2+3\ZZ$. Then $\text{disp}^+(\lambda,S)$ is the partition on the right.

\begin{figure}[h]
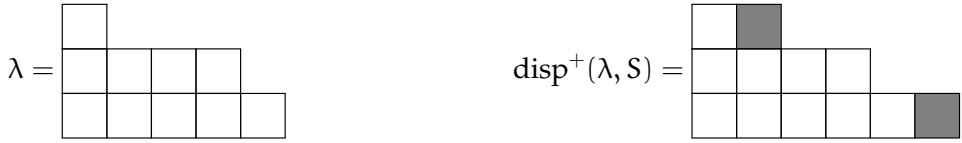

\centering
  \begin{minipage}{.4\textwidth}
   \ytableausetup{centertableaux}
$\lambda =
\begin{ytableau}
 \\
&&&\\
&&&&\\
\end{ytableau} $
  \label{fig:test1}
  \end{minipage}
  \begin{minipage}{.4\textwidth}
\ytableausetup{centertableaux}
$\text{disp}^+(\lambda,S) = 
\begin{ytableau}
\ & *(gray) \\
&&&  \\
&&&&& *(gray) \\
\end{ytableau}$
  \label{fig:test2}
  \end{minipage}

\caption{A partition $\lambda$ and its upwards displacement}
\label{fig:partition_and_displacement}
\end{figure}

\end{example}

We now define a sequence of partitions and tableaux corresponding to the symbols between $1$ and $2g-2$.
Suppose that $t(a,b) = n$, and assume that the torsion of the $n$-th loop is $s$. Denote 
\[
S_{2g-n} = S_n = (b-a) + s\ZZ.
\]
If neither $n$ nor $2g-n$ appear in $t$, we set $S_n$ and $S_{2g-n}$ to be empty. Note that this is well-defined since $t$ was assumed to be a Prym tableau. 

Now, let $t$ be a tableau on a partition $\lambda$. Define the following partitions by induction. 
\[
\lambda_0 = \emptyset,
\]
\[
\lambda_{i+1} = \text{disp}^+(\lambda_i,S_{i+1}).
\]
Define a corresponding tableau $t_{i+1}$ by filling the cells of $\lambda_{i+1}\setminus\lambda_i$  with the value $i+1$.

By construction, every square tableau contained in $t_{i}$ is Prym. Moreover, during the construction of each $t_{i}$, we do not add any new conditions on  divisors, so $ P(t) \subseteq P(t_{2g-1})$. See Example \ref{ex:sequence_of_displacements} below for a demonstration of this process.
Given a tableaux $t$, we define its \emph{dual} tableau $t^*$ by $t^*(a,b) = 2g- t(r+2-b,r+2-a)$. By construction,  we have $t_n(t^*) = t_n(t)$ for every $n$.

\begin{lemma}\label{lem:monotonicity} 
Let $n\leq g$. If $t(a,b)=n$, then $\lambda_n$ contains the cell $(a,b)$. 
If $t(a,b) = 2g-n$ then $\lambda_n$ contains $(r+2-b,r+2-a)$. In particular, if $t(a,b)=2g-n$ and $(a,b)$ is below or at the anti-diagonal, then $\lambda_{n}$ contains $(a,b)$. 
\end{lemma}

\begin{proof}
We prove the first part by induction. If $n_0\leq g$ is the lowest symbol appearing in the tableau, then $t(1,1) = n_0$, and  $t_{n_0}$ contains $(1,1)$. Assume now that the claim is true for every $i=1,\ldots,n_0$ that appears in $t$, and suppose that $t(a,b)=n$.  Since $t(a-1,b)$ and $t(a,b-1)$ are both $<n$, it follows from induction that $\lambda_{n-1}$ contains both $(a-1,b)$ and $(a,b-1)$. This implies the claim.
For the second part, if $t(a,b) = 2g-n$, then by the definition of the dual tableau, $t^*(r+2-b,r+2-a) = n$. From the first part, it follows that $\lambda_n(t^*)$ contains $(r+2-b,r+2-a)$. But $\lambda_n(t^*) = \lambda_n(t)$ for every $n$, so the claim follows.
\end{proof}

When a Prym divisor is described by a tableau, its parity is determined by the position of $g$ in the tableau (if it appears), and its rank is bounded from below by the length of the tableau minus $1$.  
Therefore, the main obstacle in proving that the parity of a divisor coincides with the parity of its rank shows up when the position of $g$ does not match the length. However, as the following proposition shows, in that situation, the rank of the corresponding divisors is higher than predicted by the length.

\begin{proposition}\label{prop:larger_tableau}
Let $t$ be a tableau of length $r+1$, and let $D\in P_\epsilon(t)$. If $\epsilon\not\equiv r\mod 2$, then $r(D)>r$. 
\end{proposition}

\begin{proof}
It suffices to show that $D\in P(s)$ for some tableau $s$ of length strictly greater than $r+1$.
Let $t(a,b) = n$ or $t(a,b)=2g-n$ with $(a,b)$ below or at the anti-diagonal. Lemma \ref{lem:monotonicity} implies that  $\lambda_n$ contains $(a,b)$. By the assumption that $\epsilon\not\equiv r\mod{2}$, together with the fact that $P_{\epsilon}\neq\emptyset$, it follows that the symbol $g$ may only appear at a cell $(a,b)$ if $b-a\equiv (r+1)\mod 2$. In particular, the anti-diagonal only contains symbols 
 $n$ or $2g-n$ with $n>0$. Lemma \ref{lem:monotonicity} now implies that  $\lambda_{g-1}$ contains all the cells that are on or below  the anti-diagonal of $\lambda$. All the cells located just above or to the right of the anti-diagonal are of the form $(a,b)$ with $a+b=r+3$, and in particular, satisfy $b-a\not\equiv (r+1)\mod 2$, so $\lambda_g$ contains all of them, including the cells  $(r+2,1)$ and $(1,r+2)$, which did not appear in $\lambda$.

The proof will be complete once we show that $\lambda_{2g-1}$ contains the entire square of length $r+2$.
To that end, we show by induction that for all $k\geq 0$,  if the symbol appears in the $(a,b)$ cell of $t_g$, then $\lambda_{g+k}$  contains the cell $(r+3-b,r+3-a)$. 
Indeed, if $a+b \geq r+3$, then $(r+3-b,r+3-a)$ is in $\lambda_g$, and  there is nothing to prove. Otherwise, both $(a,b+1)$ and $(a+1,b)$ appear in $\lambda_g$, and   since both $t_g(a+1,b)$ and $t_g(a,b+1)$ are larger than $t_g(a,b)$, it follows from induction that $\lambda_{g+{k-1}}$ contains both $(r+3-b-1,r+3-a)$ and $(r+3-b,r+3-a-1)$. Therefore,   $\lambda_{g+k}$ contains $(r+3-b,r+3-a)$. 
Since $t_g(1,1)$ is smaller or equal $g$, it follows that $\lambda_{2g-1}$ contains $(r+2,r+2)$.

\end{proof}

\begin{example}\label{ex:sequence_of_displacements}
If $t$ is the tableau from Example \ref{ex:higher_rank}, then $\lambda_{2g-1}=\lambda_7$ is precisely the tableau $s$ from the same example.

On the other hand, suppose that $g=9$, and 
\begin{center}
t = \begin{tabular}{ |c|c|c| } 
 \hline
 7 & 8 & 9 \\ \hline
 4 & 5 & 6 \\ \hline
 1 & 2 & 3 \\ 
 \hline
\end{tabular}
\end{center}
In this case, $\lambda_{2g-1}(t)=\lambda_{17}(t)=t$, and the construction does not provide any new information. This does not contradict Prop \ref{prop:larger_tableau}, because the parity of the length of $t$ matches up with the parity of the position of $g=9$.
\end{example}

\begin{theorem}\label{thm:parity}
Let $D$ be a divisor in $V^{\epsilon}$, where $\epsilon\in\{-1,0\}$. Then $r(D)\equiv \epsilon\mod 2$. In other words, the connected components of $V$ correspond to the  parity of the rank of the divisors. 
\end{theorem}
\begin{proof}
From Formula (\ref{eq:BN_locus}),  
there exists a square tableau $t$ of length $r(D)+1$ such that $D\in T(t)$. Moreover, $t$ is the largest square tableau with that property, since otherwise the rank of $D$ would be strictly greater than $r$. As $D$ is a Prym divisor, $t$ must be a Prym tableau. Assume by contradiction that the parity of $D$ differs from  $r(D) \mod{2}$. Then Proposition \ref{prop:larger_tableau} implies that there is a larger tableau $s$ such that $D\in T(s)$,
which is a contradiction. 
\end{proof}

We now make the following definition in light of Theorem \ref{thm:parity}.
\begin{definition}
Let $\pi:\widetilde\Gamma\to\Gamma$ be a folded chain of loops. For $r\geq -1$ we define the \emph{tropical Prym--Brill--Noether locus} $V^r(\Gamma,\pi)$ to be the closed subset
\begin{equation*}
V^r = V^r(\Gamma,\pi)=\big\{[D]\in\Pic_{2g-2}(\widetilde{\Gamma})\big\vert r(D)\geq r \textrm{ and }  [D]\in V^\epsilon(\Gamma, \pi)\big\}
\end{equation*}
in $V(\Gamma, \pi)\subseteq\Pic_{2g-2}(\widetilde{\Gamma})$, where $\epsilon\equiv r\mod 2$.
\end{definition}

From the proof of Theorem \ref{thm:parity}, it follows that we can describe the tropical Prym--Brill--Noether locus via Prym tableaux of the appropriate length. That is,
\begin{corollary}
The Prym--Brill--Noether locus $V^r$ satisfies 
\[
V^r(\Gamma,\pi) = \bigcup_{t} P_\epsilon(t),
\]
as $t$ varies over the Prym tableaux of length $r$, and  $\epsilon\equiv r\mod 2$. 
\end{corollary}
\noindent Note that many of the elements in the union are empty, since $P_\epsilon(t)=\emptyset$ whenever $g$ appears in $t$ at a cell $(a,b)$ with $b-a+1\neq\epsilon$.

The question remains, whether there is an intrinsic tropical characterization for the two components of $V(\Gamma,\pi)$, when $\pi$ is not the standard cover of chain of loops.  

\begin{conjecture}
Let $\pi:\widetilde\Gamma\to\Gamma$ be a topological double cover. Then one component of $V(\Gamma,\pi)$ has a dense open set of divisors of rank $-1$, and the other component has a dense open set of divisors of rank $0$. 
\end{conjecture}

Note that if $\Gamma = \Gamma_X$ for a smooth projective algebraic curve $X$, and the tropicalization map is surjective on the Prym--Brill--Noether locus (as is the case when $\Gamma$ is a chain of loops), then Baker's specialization lemma from \cite{Baker_specialization} implies that one of the components of $V(\Gamma,\pi)$ consists of only effective divisors. Moreover, if the other component contains a non-effective divisor, then it contains an open set of non-effective divisors by \cite[Theorem 4.1]{Len_BNrank}. 

\section{Prym--Brill--Noether numbers of folded chains of loops}
Having established the theory of special divisors on folded chains of loops in the previous section, we now compute the dimensions of their Prym--Brill--Noether loci. 
Throughout, we fix an integer $r$ and $\epsilon\equiv r\mod2$.



\subsection{Generic edge length}
In this subsection, we assume that $\Gamma$ is a chain of loops with generic edge length and $\pi\colon\widetilde{\Gamma}\rightarrow\Gamma$ is a folded chain of loops. 
\begin{proposition}\label{prop:maxCell}
Suppose that $g-1\geq\binom{r+1}{2}$. Then the Prym Brill--Noether locus $V^r(\Gamma,\pi)$ has a component of dimension at least $g-1-\binom{r+1}{2}$.
\end{proposition}
\begin{proof}
The proof will follow from repeatedly applying Lemma \ref{lem:equivalenToCanonical}.
Choose a square partition $\lambda$ of length $r+1$, and construct a tableau $t$ as follows. Place the symbol $g$ along the anti-diagonal (the cells with coordinate $(a,b)$ with $a+b=r+2$). Fill the part of the tableau below the anti-diagonal with integers $0<i_1< i_2<\ldots < i_{\binom{r+1}{2}}<g$, and  the part above the anti-diagonal  symmetrically according to
$t(r+2-b,r+2-a) = 2g-t(a,b)$.
By the position of the symbol $g$, along with Theorem \ref{thm:parity}, the parity of any divisor in  $P(t)$ equals the parity of $r+1$,  so $P_\epsilon(t) = P(t)$, and $P(t)\subseteq V^r(\Gamma,\pi)$.

The tableau determines the position of the chips on the loops $\widetilde{\gamma}_i$ and $\widetilde{\gamma}_{2g-i}$ for $i=1,2,\ldots, \binom{r+1}{2}$. Each such pair is mapped down by $\pi$ to a pair of chips on $\gamma_i$ that are equidistant from the vertices. Those  chips may be moved (by maintaining linear equivalence) to the two vertices of $\gamma_i$. It also determines a chip on $\widetilde{\gamma}_g$ that is mapped down to the vertex of $\gamma_g$. 

The integers $i=\binom{r+1}{2}+1, \binom{r+1}{2}+2,\ldots, g-1$ do not appear in the tableau, so  any choice for the position of the chip on $\gamma_i$ will result in a divisor of rank $r$. To guarantee that this divisor maps down to $K_{\Gamma}$ we choose  the chips on $\widetilde{\gamma}_i$ and $\widetilde{\gamma}_{2g-i}$ to be equidistant from the vertices. Their image is a pair  of chips on $\gamma_i$ that may be moved to the vertices while maintaining linear equivalence. 
By the construction, we have one degree of freedom for each such pair, and in total $g-1-\binom{r+1}{2}$ degrees of freedom. 
\end{proof}

\begin{example}
Let $r=2$, let $\Gamma$ be a chain of $g=5$ generic loops, and $\tilde\Gamma$ a chain of $9$ loops double covering it. In this case, we expect the PBN locus to be $1$ dimensional.  The tableau 
\begin{center}
\begin{tabular}{ |c|c|c| } 
 \hline
 5 & 8 & 9 \\ \hline
 2 & 5 & 7 \\ \hline
 1 & 3 & 5 \\ 
 \hline
\end{tabular}
\end{center}
gives rise to $g^2_8$'s  on $\tilde\Gamma$ in which the location of the chip on loops $1, 2,3,5,7,8,9$ is determined, and the chip on loops $4,6$ may be chosen arbitrarily (see Figure \ref{Fig:Prym_divisor}). Explicitly, there is a single pole after loop $9$, the chip on loop $1$ (resp. $9$) is on the right (resp. left) vertex, the chip on loop $2$ (resp. $8$) is one step counter clockwise (resp. clockwise) from the right vertex, the chip on loop $3$ (resp. $7$) is on the left (resp. right) vertex, and the chip on loop $5$ is at the top vertex. 
In order for those divisors to be Prym, we must choose the free chips on loops $4$ and $6$  symmetrically, so we get a $1$-dimensional family. 

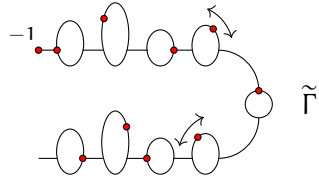
\begin{figure}[h!]
\begin{tikzpicture}[scale=.6]

\begin{scope}

\draw (6,-0.75) node {$\widetilde\Gamma$};

\draw (0.41,0.35) -- (0,0.35);
 \draw (0.7,0.5) ellipse (0.3 and .5);
 \draw (0.98,0.35) -- (1.45,0.35);
 \draw (1.7,0.7) ellipse (0.3 and .7);
 \draw (1.95,0.35) -- (2.42,0.35);
 \draw (2.7,0.4) ellipse (0.3 and .4);
  \draw (2.98,0.35) -- (3.42,0.35);
 \draw (3.7,0.45) ellipse (0.3 and .45);

\draw (3.98,0.35) arc[radius = .9, start angle=90, end angle=0];
 \draw (4.88,-.85) circle (0.3);
 \draw (3.98,-2.05) arc[radius = .9, start angle=-90, end angle=0];

  \draw [fill=red] (0,0.35) circle (0.7mm);
  \draw [] (-.35,0.7) node {\tiny $-1$};
  \draw [fill=red] (0.41,0.35) circle (0.7mm);
    \draw [fill=red] (1.45,1.05) circle (0.7mm);
    \draw [fill=red] (3,0.35) circle (0.7mm);
    \draw [fill=red] (4.88,-0.55) circle (0.7mm);
   
    \draw [fill=red] (3.87,0.82) circle (0.7mm);
     \draw [<->] (4.28,0.62) arc[radius = 1, start angle=15, end angle=65];

\end{scope}

 \begin{scope} [shift={(0,-2.4)}]

\draw (0.41,0.35) -- (0,0.35);
 \draw (0.7,0.5) ellipse (0.3 and .5);
 \draw (0.98,0.35) -- (1.45,0.35);
 \draw (1.7,0.7) ellipse (0.3 and .7);
 \draw (1.95,0.35) -- (2.42,0.35);
 \draw (2.7,0.4) ellipse (0.3 and .4);
  \draw (2.98,0.35) -- (3.42,0.35);
 \draw (3.7,0.45) ellipse (0.3 and .45);


  \draw [thin, color=black, fill=red] (0.98,0.35) circle (0.7mm);
    \draw [fill=red] (1.95,1.05) circle (0.7mm);
    \draw [fill=red] (2.4,0.35) circle (0.7mm);
    
    \draw [fill=red] (3.53,0.82) circle (0.7mm);
     \draw [<->] (3.12,0.62) arc[radius = 1, start angle=165, end angle=115];

\end{scope}

\end{tikzpicture}
\caption{Prym divisor on a chain of 9 loops.}
\label{Fig:Prym_divisor}
\end{figure}

\end{example}



As we shall now see, the cell constructed in the last proposition is, in fact, maximal. 
\begin{lemma}\label{lem:maximize}
Let $P_\epsilon(t)$ be a cell of the Prym--Brill--Noether locus $V^r(\Gamma, \pi)$ corresponding to a Prym tableau $t$. Then $\dim{P(t)}\leq g-1-\binom{r+1}{2}$.
\end{lemma}

\begin{proof}
From Lemma \ref{lem:equivalenToCanonical}, if $i$ appears in a tableau but $2g-i$ does not, the position of the free chip on $\widetilde{\gamma}_{2g-i}$ is determined by the position of the chip on $\widetilde{\gamma}_i$. It follows that the dimension of the Prym--Brill--Noether locus is the number of $0<i<g$ such that neither $i$ nor $2g-i$ appear in the tableau. In particular, the co-dimension is bounded from below by half the number of symbols, other than $g$, that appear in the tableau. 
Since $g$ is the only symbol that may repeat in $t$, and may appear at most $r+1$ times, this number is minimized 
precisely when $g$ appears along the anti-diagonal and $t(a,b) = (2g-b,2g-a)$ away from the anti-diagonal, in which case the dimension of the Prym--Brill--Noether locus is $g-1-\binom{r+1}{2}$. 
\end{proof}

We may now prove the main result of this subsection.
\begin{theorem}\label{thm_tropicalPrymBrillNoether}
Suppose that $\pi:\widetilde\Gamma\to\Gamma$ is a folded chain of loops with generic edge length. 
\begin{enumerate}[(i)]
\item If  $g-1<\binom{r+1}{2}$, then  $V^r(\Gamma,\pi)$ is empty. 
\item When $g-1\geq\binom{r+1}{2}$, then $V^r(\Gamma,\pi)$ has pure dimension $g-1-\binom{r+1}{2}$.
\end{enumerate}
\end{theorem}
\begin{proof}
The existence part of the theorem and the dimension of the largest component follow directly from Lemma \ref{lem:maximize} and Proposition \ref{prop:maxCell}. It is left to show that the locus is pure dimensional. Namely, that every component of non-maximal dimension is contained in a component of dimension $g-1-\binom{r+1}{2}$. 

The set $V^r$ is the union of $P_\epsilon(t)$, as $t$ varies over Prym tableau of length $r+1$. Let $t$ be such a tableau. We claim that there is a tableau $s$ such that $P_\epsilon(t)\subseteq P_\epsilon(s)$, and $\dim\big(P_\epsilon(s)\big)$ is maximal. Indeed, consider the tableau $t_{g-1}(t)$, as constructed in Definition \ref{def:upwards_displacement}. Then $t_{g-1}(t)$ consists only of the symbols $1,\ldots,g-1$, and contains the lower triangle of length $r$. 
Let $s$ be the tableau of length $r+1$, that coincides with $t_{g-1}$ below the anti-diagonal, contains the symbol $g$ along the anti-diagonal, and $s(r+2-b,r+2-a) = 2g - s_{g-1}(a,b)$ above it. 
Then $P_\epsilon(t)\subseteq P_\epsilon(s)$,  and since none of the symbols apart from $g$ may repeat in the tableau, we conclude that $\dim{P_\epsilon(s)} = g-1-\binom{r+1}{2}$. 
\end{proof}

In the $0$-dimensional case, we find a tropical analogue of a classical result by De Concini and Pragacz \cite[Theorem 9]{DeConciniPragacz}.

\begin{corollary}
When $g-1=\binom{r+1}{2}$, the number of Prym--Brill--Noether divisors is \[
\frac{\binom{r+1}{2}!}{(2r-1)\cdot(2r-3)^2\cdot\ldots\cdot 1^r}.
\]
\end{corollary}
\begin{proof}
By Theorem \ref{thm_tropicalPrymBrillNoether}, the Prym--Brill--Noether divisors correspond to symmetric $(r+1)\times(r+1)$ Young tableaux with the symbol $g$ along the anti-diagonal. Each such tableau is uniquely determined by a Young tableau  with row lengths $(r,r-1,\ldots,2,1)$. The result now follows from the hook length formula.
\end{proof}


\subsection{Special chains of loops}\label{section_specialchainofloops}
We now turn to the case where the base graph $\Gamma$ is a chain of loops in which the torsion of every loop is $k$. As for $\widetilde\Gamma$, the torsion of $\widetilde{\gamma}_i$ is $k$ when $i\neq g$, and the torsion remains $2$ when $i=g$. 
In the corresponding tableau, a symbol $i\neq g$ is allowed to repeat, but only if the lattice distance between every two occurrences is a multiple of $k$.  


%
%
As in the introduction, denote 
 \[
 n(r,\ell) =
  \begin{cases}
                    \binom{\ell+1}{2}+\ell(r-\ell) &
                    \text{if $\ell\leq r-1$} \\
                    \binom{r+1}{2} & \text{if $\ell > r-1$},
  \end{cases}
  \]
where $\ell=\lceil\frac{k}{2}\rceil$. Note that for even $k\leq 2r-2$  we have $n(r,\ell)=\frac{rk}{2}-\frac{k^2}{8}+\frac{k}{4}$.


In what follows, we use the term \emph{lower triangular tableau} to describe a tableau consisting of cells with coordinates $(a,b)$ with $a+b<r+2$. Such a tableau is obtained by restricting a square tableau to the cells below the anti-diagonal (see Figure \ref{fig:upperTriangularTableau}).


\begin{lemma}\label{lem:upperTriangularTableau}
Let $k$ be a non-negative integer that is either even or greater than $2r-2$. Then the smallest number of symbols in a  $k$-uniform lower triangular tableau is $n(r,\ell)$.
\end{lemma}
\begin{proof}
If $k> 2r$, then no repetition is allowed in the tableau, and it contains at most $\binom{r+1}{2}$ symbols. So we may assume $k\leq 2r$. 

We first show that there exists a tableau with $\frac{rk}{2}-\frac{k^2}{8}+\frac{k}{4}$ symbols as follows (cf. \cite[Lemma 3.5]{Pflueger_kgonalcurves}).
Begin with a square $(r+1)\times (r+1)$ partition $\lambda'$. Place the integers from $1$ to $\frac{k}{2}$ sequentially along the first column starting from the bottom. Repeat this in the second column starting from $\frac{k}{2}+1$, and continue until the integers between $1$ and $\frac{k}{2}(r+1)$ fill the first $\frac{k}{2}$ rows. 
We fill the subsequent rows inductively, by assigning  $t'(a,b+\frac{k}{2}) = t'(a,b) + \frac{k^2}{4}$. The lower triangular tableau $t$ will be the restriction of $t'$ to the cells below the anti-diagonal. See Figure \ref{fig:upperTriangularTableau} for an example. 

One now checks that the number of symbols in the first $\frac{k}{2}$ rows of $t$ equals $n(r,\ell)$. Moreover,  we claim
that none of  the subsequent rows introduces new symbols. To see that, let $(a,b)$ be a cell of $\lambda$. Let $\alpha\in\NN$ be such that $a-\alpha\cdot\frac{k}{2}$ is between $0$ and  $\frac{k}{2}$. Then $(a-\alpha\cdot\frac{k}{2}, b+\alpha\cdot\frac{k}{2})$ is a cell of $\lambda$ in the first $\frac{k}{2}$ rows, and by construction, $t(a-\alpha\cdot\frac{k}{2}, b+\alpha\cdot\frac{k}{2}) = t(a,b)$.
It follows that  the total number of symbols in $t$ is $n(r,\ell)$.

We now show that no $k$-uniform lower triangular tableau has fewer than  $n(r,\ell)$  symbols. 
Let $t$ be such a tableau, and let $S_u$ consist of the cells that are at most $\frac{k}{2}-1$ steps directly to the right of the diagonal, or at most $\frac{k}{2}$ steps directly above the diagonal. 
Since $t$ is $k$-uniform, any two cells in $S_u$ must have different symbols, so the size of $S_u$ is a lower bound for the number appearing in any such tableau. To determine the size of $S_u$, we count the number of symbols directly to the right and directly above each diagonal cell. Assume first that $r$ is even and $k\equiv 2\mod 4$. Each of the first $\frac{r+1}{2}-\frac{k}{4}$ diagonal cells  contributes $k$ cells. The next diagonal cell contributes 3 fewer cells, and the number goes down by $4$ for each subsequent diagonal cell. The last cell contributes $3$. A straightforward calculation shows that the sum is exactly $n(r,\ell)$.
A similar argument works for other values of $k$ and $r$, as long as $k$ is even.
\end{proof}

\begin{figure}
\ytableausetup{centertableaux}
\begin{ytableau}
9\\
6&8\\
5&7&9\\
2&4&6&8\\
1&3&5&7&9
\end{ytableau}
\caption{The lower triangular tableau constructed in Lemma \ref{lem:upperTriangularTableau} when $r=5, k=4$.}
\label{fig:upperTriangularTableau}
\end{figure}

\begin{corollary}\label{cor_tropPBNwithgonality}
Assume that $k$ is either even or greater than $2r-2$. 
\begin{enumerate}[(i)]
\item
If $g-1<n(r,\ell)$ then the Prym--Brill--Noether locus is empty. 
\item Otherwise, the dimension of the Prym--Brill--Noether locus is $\rho: = g-1 - n(r,\ell)$.
\end{enumerate}
\end{corollary}
\begin{proof}
Assume that $g-1\geq n(r,\ell)$.
Let $t$ be the lower triangular tableau constructed in  Lemma \ref{lem:upperTriangularTableau}. We complete it to a maximally Prym tableau, similarly to Lemma \ref{lem:maximize}, by setting  $t(c,d) + t(a,b) = 2g$ whenever $(a,b)+(c,d)=(r+2,r+2)$, and $t(a,b) = g$ along the anti-diagonal. The dimension of the corresponding cell is now $\rho$. It remains to show that no cell has dimension greater than $\rho$.

Indeed, let $t$ be any tableau, and let  $S$ consist of the cells that are at most $\frac{k}{2}-1$ steps directly to the right of the anti-diagonal, or  $\frac{k}{2}$ steps directly below the anti-diagonal. $S$ is similar to the set $S_u$ constructed in Lemma \ref{lem:upperTriangularTableau}, except that it is a subset of a square rather than a triangular tableau. $S$ intersects the anti-diagonal at $\frac{k}{2}$ cells, and its restrictions to  the upper and lower triangular parts of $t$ each consists of $|S_u|$ symbols. Therefore, we have $|S| = 2\cdot |S_u| + \frac{k}{2}$. 

As always, the dimension of the corresponding cell is bounded from above by half the number of symbols other than $g$ in the tableau. The symbol $g$ may appear on or away from the anti-diagonal, but in any case it may repeat at most $\frac{k}{2}$ times. Any other symbol may not repeat in $S$. Therefore, $S$ contains at least  $|S|-\frac{k}{2} = 2\cdot|S_u|$ symbols distinct from $g$. It follows that the co-dimension of the corresponding cell is at most $|S_u| = n(r,\ell)$. 

Finally, if $g-1<n(r,\ell)$, then the argument above shows that there is no Prym tableau of length $r+1$ on the symbols $1,2,\ldots ,2g-1$, and thus the Prym--Brill--Noether locus is empty. 
\end{proof}


\begin{remark}
The dimension of $V^r(\Gamma,\pi)$ for odd gonality is determined in \cite[Theorem A]{CLRW_PBN}. Moreover, via a refined study of the polyhedral structure of $V^r(\Gamma,\pi)$, it is shown that the locus is pure dimensional in any gonality, and connected in co-dimension $1$ whenever its dimension is positive \cite[Theorem C]{CLRW_PBN}. 
\end{remark}




\section{Proof of Theorem  \ref{thm_PrymBrillNoetherwithfixedgonality}}




The following Lemma \ref{lemma_liftingcovers} will allow us to derive properties of  $k$-gonal algebraic Prym curves from $k$-gonal tropical Prym  curves. 

\begin{lemma}\label{lemma_liftingcovers}
Let $k\geq 2$. Suppose that $K$ is a non-Archimedean field whose residue field has characteristic prime to $2$ and $k$. 
Let $\pi:\widetilde\Gamma\to\Gamma$ be a harmonic double cover such that $\Gamma$ is a metric graph of genus $g$, and let $\kappa:\Gamma\to\ T$ be a harmonic $k$-fold cover of a segment  $T$. Then there is an unramified double cover $f:\widetilde{X}\to X$, such that $\trop(f)=\pi$, and $X$ is a $k$-gonal curve of genus $g$. 
\end{lemma}
\begin{proof}
We begin by promoting $\kappa$ to a map of metrized complexes as follows. Let $\mathcal{X}$ be the metrized complex obtained by attaching a copy of $\PP^1$ at every vertex of $\Gamma$, and $\mathcal{T}$ the metrized complex obtained from $T$ by attaching a rational component at the image of every vertex. Fix a vertex $v$  of $\Gamma$,   let $p$ be its image in $T$, and let $X_v$ and $X_p$ be the corresponding rational components. Let $t,s$ be the tangent directions emanating from $p$, and $t_1,\ldots, t_m, s_1,\ldots s_n$ be the corresponding tangent directions at $v$ with dilation factors $a_1,\ldots a_m, b_1,\ldots, b_n$. Note that the sum of the $a_i$'s equals the sum of the $b_j$'s due to harmonicity.

Assume that $t$ and $s$ correspond to the points $0,\infty$ of $X_p$. Let $T_1,\ldots, T_m$ be the points of $X_p$ corresponding to $t_1,\ldots, t_m$, and  $S_1,\ldots, S_n$ the points corresponding to $s_1,\ldots, s_n$.  
Let $f$ be  the rational function with $a_i$ zeroes at each $T_i$ and $b_j$ poles at $S_j$. Then $f$ induces a cover  $X_v\to X_p$ with  ramification data given by $\kappa$. Repeating this construction for each vertex, we obtain a map of metrized complexes that specializes to $\kappa$. By construction, the genus of $\mathcal{X}$ equals the genus of $\Gamma$. 
By the assumption on the characteristic,  $\kappa$ is tame. 
 By \cite[Lemma 7.15]{ABBRI}, there is a map $X\to\PP^1$ of smooth curves tropicalizing to $\kappa$, where the genus of $X$ equals the genus of $\Gamma$.

Finally, since  $\Gamma$ is weightless, \cite[Lemma 5.9]{JensenLen_thetachars} implies that  the map $\pi$ may be lifted to an unramified double cover $\widetilde{X}\to X$. 
\end{proof}

\begin{proof}[Proof of Theorem \ref{thm_PrymBrillNoetherwithfixedgonality}]
Let $r\geq -1$ and $k$ be either even or $\geq 2r$. 
Our task is to produce an 
unramified double cover $\pi\colon \widetilde{X}\rightarrow X$ of a smooth projective curve of genus $g$ that is generic in an open subset of the $k$-gonal locus of $\calR_g$  such that the inequality
\begin{equation*}
\dim V^r(X,\pi)\leq g-1 -n(r,\ell) \ .
\end{equation*}
holds.
We choose $\pi\colon \widetilde{X}\rightarrow X$ to be a one-parameter-smoothing over a non-Archimedean field $K$ of the unramifed double cover between two chains of loops $\pi^{trop}\colon\widetilde{\Gamma}\rightarrow \Gamma$ as discussed in Section \ref{section_specialchainofloops} above. Finding such a double cover is always possible by Lemma \ref{lemma_liftingcovers}. 

By Theorem \ref{thm_skeletonofPrym=Prymofskeleton} we can use Baker's specialization inequality \cite[Corollary 2.11]{Baker_specialization} and obtain:
\begin{equation*}
\Trop\big(V^r(X,\pi)\big) \subseteq V^r(\Gamma_X,\pi^{trop}) \ .
\end{equation*} 
Since both $\Gamma$ and $\Gamma'$ are trivalent and without vertex-weights, both of their Jacobians (and therefore also the Prym-variety $\Pr(X,\pi)$) are maximally degenerate. Therefore we may apply Gubler's Bieri-Groves Theorem for maximally degenerate abelian varieties  \cite[Theorem 6.9]{Gubler_trop&nonArch} and Theorem \ref{cor_tropPBNwithgonality} above to find:
\begin{equation*}
\dim V^r(X,\pi)=\dim \Trop\big(V^r(X,\pi)\big)\leq \dim V^r(\Gamma_X,\pi^{trop}\big)=g-1-n(r,\ell) \ .
\end{equation*}
If $g-1<\binom{r+1}{2}$, the tropical Prym--Brill--Noether locus $V^r(\Gamma, \pi^{trop})$ is empty and so also the algebraic Prym--Brill--Noether locus  $V^{r}(X,\pi)$ is empty. 
\end{proof}

\bibliographystyle{amsalpha}
\bibliography{biblio}{}

\end{document}